\documentclass[12pt]{amsart}
\usepackage{amssymb}
\usepackage{amsmath}
\usepackage{csquotes}
\usepackage[pagebackref=false,colorlinks=true,urlcolor  = blue,linkcolor=blue,citecolor=blue]{hyperref}
\usepackage{mathrsfs,graphicx,amsthm, mathtools}
\usepackage{multirow}
\usepackage{makecell}
\usepackage[all]{xy}
\usepackage{tikz}
\usepackage[noadjust]{cite}

\usepackage{xcolor}
\input xy
\xyoption{all}
\definecolor{gree}{RGB}{255,91,17}
\definecolor{Gr}{RGB}{31,178,38}

\newtheorem{lemma}{Lemma}[section]
\newtheorem{corollary}[lemma]{Corollary}
\newtheorem{theorem}[lemma]{Theorem}
\newtheorem{proposition}[lemma]{Proposition}

\theoremstyle{definition}
\newtheorem{remark}[lemma]{Remark}
\newtheorem{definition}[lemma]{Definition}

\newtheorem{example}[lemma]{Example}

\newtheorem{conjecture}[lemma]{Conjecture}
\newtheorem{question}[lemma]{Question}
\DeclareMathOperator{\modd}{mod}
\DeclareMathOperator{\End}{End}
\DeclareMathOperator{\add}{add}
\DeclareMathOperator{\Hom}{Hom}
\DeclareMathOperator{\Ext}{Ext}
\DeclareMathOperator{\ind}{ind}
\DeclareMathOperator{\mcm}{\mathcal{M}}

\newcommand{\gd}{\mathrm{gl.dim}\,}

\newcommand{\md}{\mathcal{D}}

\newcommand{\uu}{\mathscr{U}}

\newcommand{\mcc}{\mathscr{C}}

\newtheorem{mainpro}{Proposition}

\newtheorem{mainthmA}{Theorem}

\newtheorem{mainthmB}{Theorem}

\newtheorem{mainthmD}{Theorem}

\newtheorem{mainthmE}{Theorem}

\newtheorem{mainthmF}{Theorem}

\newtheorem{mainthmG}{Theorem}

\newtheorem{mainthmH}{Theorem}

\definecolor{blac}{RGB}{0,0,1}
\begin{document}
\title[Coxeter matrices and homological quadratic forms]{Coxeter matrices and homological quadratic forms of $\boldsymbol{n}$-hereditary algebras}
\author{Raziyeh Diyanatnezhad}
\address{Department of Pure Mathematics\\
Faculty of Mathematics and Statistics\\
University of Isfahan\\
P.O. Box: 81746-73441, Isfahan, Iran and School of Mathematics, Institute for Research in Fundamental Sciences (IPM), P.O. Box: 19395-5746, Tehran, Iran}
\email{r.diyanat@sci.ui.ac.ir}

\author{Alireza Nasr-Isfahani}
\address{Department of Pure Mathematics\\
Faculty of Mathematics and Statistics\\
University of Isfahan\\
P.O. Box: 81746-73441, Isfahan, Iran and School of Mathematics, Institute for Research in Fundamental Sciences (IPM), P.O. Box: 19395-5746, Tehran, Iran}
\email{nasr$_{-}$a@sci.ui.ac.ir / nasr@ipm.ir}



\begin{abstract}
We study the Coxeter matrices and the homological quadratic forms of $n$-hereditary algebras within the framework of higher dimensional Auslander--Reiten theory. Let $\Lambda$ be a finite dimensional $n$-hereditary algebra with the Coxeter matrix $\Phi$ and the homological quadratic form $\chi$. We prove that if $\Lambda$ is $n$-representation finite, then there exists a positive integer $d$ such that $\Phi^d=1$. In the case $n$ is an odd number, we show that if there exists a positive integer $d$ such that $\Phi^d=1$, then $\Lambda$ is $n$-representation finite. Let $\mathcal{C}^0$ be the subcategory of $\modd\Lambda$ which is a higher analogue of the module category in the context of higher dimensional Auslander--Reiten theory. We introduce a Grothendieck group $\mathrm{K}_0(\mathcal{C}^0)$ associated with $\mathcal{C}^0$ and show that it is isomorphic to the Grothendieck group of $\Lambda$. We further prove that if the restriction of $\chi$ to $\mathrm{K}_0(\mcc^0)$ is positive definite, then $\Lambda$ is $n$-representation finite for odd $n$. To prove these results, we first show that indecomposable $n$-preprojective and $n$-preinjective modules are uniquely determined up to isomorphism by their dimension vectors for odd $n$. We also provide examples of $n$-representation finite algebras that the restriction of $\chi$ to $\mathrm{K}_0(\mcc^0)$ is not positive definite. 
\end{abstract}
\maketitle
\tableofcontents
\section{Introduction}
Higher Auslander--Reiten theory was first introduced by Iyama in \cite{I1} and was further developed in a series of subsequent works \cite{I2,I4,I3}. Beyond its deep connections to the representation theory \cite{HI2,IO,M}, it has found significant connections to diverse areas of mathematics, including combinatorics \cite{OT,HJ,W}, algebraic geometry \cite{HI1, HIMO, JKM, IW,IW1,IW2}, the categorification of cluster algebras \cite{GLS}, as well as higher category
theory and algebraic $K$-theory \cite{DJY}. Given a positive integer $n$, Iyama introduced the notion of $n$-cluster tilting subcategories of abelian categories to establish a higher-dimensional analogue of the classical Auslander correspondence. Let $\Lambda$ be a finite dimensional algebra of the global dimension at most $n$ and $\modd\Lambda$ be the category of finitely generated right $\Lambda$-modules. If $\modd\Lambda$ admits an $n$-cluster tilting subcategory with an additive generator, then $\Lambda$ is called an \textit{$n$-representation finite} algebra. These algebras were introduced by Iyama and oppermann in \cite{IO} and have been extensively studied by several authors \cite{I4,HI1,HI2,IO1,HZ1,HZ2,HZ3}. Notably, the case $n=1$ recovers precisely the class of representation finite hereditary algebras. In the classical theory, representation infinite algebras emerge as a natural counterpart to representation finite ones. Inspired by this fact, Herschend, Iyama and Oppermann introduced the concept of $n$-representation infinite algebras in \cite{HIO}. Let $\nu$ be the Nakayama functor on the bounded derived category $\md_\Lambda$ of $\modd \Lambda$. Then, the autoequivalences $\nu_n\coloneqq \nu\circ[-n]$ and $\nu^{-1}_n\coloneqq \nu^{-1}\circ[n]$ are well-defined on $\md_\Lambda$ (see \cite{I4}).  A finite dimensional algebra $\Lambda$ with the global dimension at most $n$ is said to be \textit{$n$-representation infinite} if $\nu_n^i(\Lambda)$ is concentrated in degree $0$ for all $i \geq 0$. Note that $1$-representation infinite algebras are exactly representation infinite hereditary algebras. In \cite{HIO}, the authors also introduced another class of algebras with the global dimension at most $n$ which called \textit{$n$-hereditary} algebras (see Definition \ref{n-h} in Preliminaries). $1$-hereditary algebras are precisely hereditary algebras. Herschend et al. in \cite{HIO} showed that every $n$-hereditary algebra is either $n$-representation finite or $n$-representation infinite. 

The representation theory of hereditary algebras is one of the most important part in the representation theory of finite dimensional algebras \cite{G,BGP,AP,Ri,Ri1,Ri2,Ke}. There are several tools to determine the representation type of hereditary algebras, for example the shape of Gabriel quivers, homological quadratic forms and Coxeter matrices (see \cite{G,AP}). A key result states that a hereditary algebra $\Lambda$ is representation finite if and only if its Coxeter matrix $\Phi$ satisfies $\Phi^d = 1$, for some positive integer $d$. We recall that for a basic finite dimensional algebra $\Lambda$ of the finite global dimension, the Coxeter matrix of $\Lambda$ is defined as $\Phi\coloneqq -C_\Lambda^t C_\Lambda^{-1}$, where $C_\Lambda$ denotes the Cartan matrix of $\Lambda$. Another well-known result is the fact that a hereditary algebra $\Lambda$ is representation finite if and only if its associated homological quadratic form is positive definite. In this case, the associated homological quadratic form has finitely many positive roots and there is a bijection between the set of isomorphism classes of indecomposable $\Lambda$-modules and the set of positive roots of the quadratic form. 

Let $K$ be a field and $\Lambda$ be a finite dimensional $K$-algebra. $M\in\modd\Lambda$ is called a \textit{brick} if $\End_\Lambda(M)$ is a division $K$-algebra. If there are only finitely many isomorphism classes of
bricks in $\modd \Lambda$, then $\Lambda$ is called  \textit{brick--finite}. Obviously any representation finite algebra is brick--finite, but the converse is not true in general (see \cite{Mi}). It is known that a finite dimensional hereditary $K$-algebra $\Lambda$ over an algebraically closed field $K$ is brick--finite if and only it is representation finite (see \cite[Section $\rm{VII}.6$]{SY}). A finite dimensional $K$-algebra $\Lambda$ is brick--finite if and only if it is $\tau$-tilting finite (for more details see \cite{DIJ}). Nowadays the study of brick--finite algebras is one of the active research area in the representation theory of algebras.  

The aim of this article is to investigate the higher dimensional variant of these benchmarks in order to ascertain the $n$-representation finiteness of $n$-hereditary algebras.

For an $n$-hereditary algebra $\Lambda$, the subcategory $\mcc^0$ of $\modd\Lambda$ plays an important role (for the definition of $\mcc^0$, see Preliminaries). This subcategory can be viewed as a higher analogue of the module category. We first prove $n$-Auslander--Reiten duality for $\mcc^0$ which has an essential role in this paper. Note that Iyama in \cite[Theorem 2.3.1]{I1} proved $n$-Auslander--Reiten duality for $n$-cluster tilting subcategories of $\modd\Lambda$.

\begin{mainthmA}$($Theorem \ref{ARd1}$)$\label{thm:MainA}
Let $\Lambda$ be an $n$-hereditary algebra. Then, there exist the following
 functorial isomorphisms for any $X, Y\in\mcc^0$ and $Z\in\modd\Lambda$.
$$
\underline{\Hom}_\Lambda(X,Z)\cong D\Ext^n_\Lambda(Z,\tau_nX)\quad\,\,\text{and}\quad\,\,
\overline{\Hom}_\Lambda(Z,Y)\cong D\Ext^n_\Lambda(\tau_n^-Y,Z).$$
\end{mainthmA}

By using $n$-Auslander--Reiten duality for $\mcc^0$, we give a characterization of $n$-representation finite algebras by bricks in $\mcc^0$.

\begin{mainthmB}$($Theorem \ref{bricks}$)$\label{thm:MainB}
Let $K$ be an algebraically closed field and $\Lambda$ be an $n$-hereditary $K$-algebra with the acyclic Gabriel quiver.
 Then $\Lambda$ is $n$-representation finite if and only if there are finitely many isomorphism classes of bricks in $\mcc^0$.
\end{mainthmB}

Let $\Lambda$ be an $n$-hereditary algebra. The Coxeter matrix of $\Lambda$ is given by $\Phi\coloneqq (-1)^nC_\Lambda^t C_\Lambda^{-1}$ (for $n$-representation finite and $n$-representation infinite cases, see \cite[Definition 2.1]{M} and \cite[Observation 4.12]{HIO}, respectively).
As a key objective of this paper, we use the Coxeter matrix to establish a criterion determining when an $n$-hereditary algebra is $n$-representation finite. For an $n$-representation finite algebra $\Lambda$, we find a positive integer $d$ such that $\Phi^d=1$. 
More precisely, we prove the following proposition.
\begin{mainpro}$($Proposition \ref{fi-re}$)$\label{mainproA}
	Let $\Lambda$ be an $n$-hereditary algebra. If $\Lambda$ is $n$-representation finite, then there exists a positive integer $d$ such that $\Phi^d=1$.                                     
\end{mainpro}
For the proof of the converse of the above proposition, we need to show that for an $n$-hereditary algebra $\Lambda$ indecomposable modules in the subcategory $\add\{\tau_n^{i}(D\Lambda)\,|\,i\geq 0\}$ of $\modd\Lambda$ are uniquely determined up to isomorphism by their dimension vectors. To show this, we consider subcategories
$$\mathscr{P}\coloneqq \add\{\tau_n^{-i}(\Lambda)\,|\,i\geq 0\}\quad\text{and}\quad
	\mathscr{I}\coloneqq \add\{\tau_n^{i}(D\Lambda)\,|\,i\geq 0\}$$
of $\modd\Lambda$ and we prove the following crucial theorem.
\begin{mainthmD}$($Theorem \ref{dv}$)$\label{thm:mainB}
Let $\Lambda$ be an $n$-hereditary algebra with odd $n$. Then the following hold.
\begin{enumerate}
	\item[(a)]
	For any two indecomposable objects $X,Y$ in $\mathscr{P}$, $X\cong Y$ if and only if $\underline{\dim}X=\underline{\dim}Y$.
	\item[(b)]
	For any two indecomposable objects $X,Y$ in $\mathscr{I}$, $X\cong Y$ if and only if $\underline{\dim}X=\underline{\dim}Y$.
\end{enumerate}  
\end{mainthmD}
In the representation theory of artin algebras, a central problem is when modules are determined by their composition factors up to isomorphism. Auslander and Reiten showed that under certain conditions modules are determined by their composition factors \cite{AR1}. In particular, by Auslander and Platzeck's work \cite{AP}, indecomposable preprojective and preinjective modules over hereditary algebras are uniquely determined by their composition factors. For an $n$-hereditary algebra $\Lambda$, the subcategories $\mathscr{P}$ and $\mathscr{I}$ can be viewed as higher dimensional analogues of the classical preprojective and preinjective components of hereditary algebras, respectively. Thus, Theorem \ref{thm:mainB} in the case $n=1$ recovers this classical result for hereditary algebras. As a corollary of Theorem \ref{thm:mainB}, we complete Proposition \ref{mainproA} in the following way.
\begin{mainthmE}$($Theorem \ref{fi-fi}$)$\label{thm:mainC}
		Let $\Lambda$ be an $n$-hereditary algebra with odd $n$. Then there exists a positive integer $d$ such that $\Phi^d=1$ if and only if $\Lambda$ is $n$-representation finite.
\end{mainthmE}
Let $\mathrm{K}_0(\modd\Lambda)$ denotes the Grothendieck group of $\Lambda$ and for any $X\in\modd\Lambda$, $[X]$ denotes the corresponding element in $\mathrm{K}_0(\modd\Lambda)$.
For a finite dimensional $K$-algebra $\Lambda$ of finite global dimension, there exists a well-defined bilinear form
$B:\mathrm{K}_0(\modd\Lambda)\times\mathrm{K}_0(\modd\Lambda)\to\mathbb{Z}$ given by $B([X],[Y])=\sum_{i\geq 0}(-1)^i\dim_K\Ext^i_\Lambda(X,Y)$. The corresponding quadratic form $\chi([X])\coloneqq B([X],[X])$ is called the \textit{homological quadratic form} of $\Lambda$. This form was first introduced by Ringel \cite{Ri}. In section \ref{g}, we introduce a Grothendieck group $\mathrm{K}_0(\mathcal{C}^0)$ associated with $n$-hereditary algebra $\Lambda$ and the subcategory $\mathcal{C}^0$ of $\modd\Lambda$ by using exact sequences with $n+2$ terms in $\mcc^0$. Since any exact sequence with $n+2$ terms in $\mcc^0$ is also an exact sequence in $\modd\Lambda$, we can restrict the bilinear form $B$ and the corresponding homological quadratic form $\chi$ of $\Lambda$ to $\mathrm{K}_0(\mcc^0)$. Then we prove the following theorem.
\begin{mainthmF}$($Theorem \ref{form}$)$\label{thm:mainD}	
	Let $\Lambda$ be an $n$-hereditary algebra with odd $n$. If the restriction of the homological quadratic form $\chi$ to $\mathrm{K}_0(\mcc^0)$ is positive definite, then $\Lambda$ is $n$-representation finite.
\end{mainthmF}
For any $X\in\mcc^0$, let $[X]_{\mcc^0}$ denote the corresponding element in $\mathrm{K}_0(\mcc^0)$. In Theorem \ref{thm:mainD}, we use $[X]$ instead of $[X]_{\mcc^0}$ since we prove the following theorem.
\begin{mainthmG}$($Theorem \ref{T1}$)$\label{thm:mainE}
	Let $\Lambda$ be an $n$-hereditary algebra. Then, $\mathrm{K}_0(\mcc^0)\cong\mathrm{K}_0(\modd\Lambda)$.
\end{mainthmG}
If $\Lambda$ is an $n$-representation finite algebra, then $\mcc^0$ is the unique $n$-cluster tilting subcategory of $\modd\Lambda$ (see Theorem \ref{th1} in Preliminaries) and hence also an $n$-abelian category (see \cite[Theorem 3.16]{J}). Grothendieck groups of $n$-abelian categories have been introduced in \cite{Re}. Furthermore, it has been shown that for an $n$-representation finite algebra $\Lambda$, the Grothendieck group of the $n$-abelian category $\mcc^0$ is isomorphic to $\mathrm{K}_0(\modd\Lambda)$ (see \cite[Theorem 3.7]{Re} or \cite[Theorems 3.9 and 3.11]{DN1}). Thus, Theorem \ref{thm:mainE} generalizes this isomorphism to $n$-hereditary algebras. To prove Theorem \ref{thm:mainD}, we apply Theorem \ref{thm:mainB} together the following fact.
\begin{mainthmH}$($Theorem \ref{form0}$)$\label{thm:mainF}
	Let $\Lambda$ be an $n$-hereditary algebra with the complete set $\{P_1,\dots,P_m\}$ and $\{I_1,\dots,I_m\}$ of the isomorphism classes of indecomposable projective and indecomposable injective $\Lambda$-modules, respectively. Then the following are equivalent.
	\begin{enumerate}
		\item[(a)]
		$\Lambda$ is $n$-representation finite.
		\item[(b)]
		$\{X\in\ind\mathscr{P}\,|\,\chi([X])=l\}$ is a finite set for any positive integer $l$.
		\item[(c)]
		$\{X\in\ind\mathscr{P}\,|\,\chi([X])=\chi([P_i])\}$ is a finite set for any $i\in\{1,\dots,m\}$.
		\item[(d)]
		$\{X\in\ind\mathscr{I}\,|\,\chi([X])=l\}$ is a finite set for any positive integer $l$.
		\item[(e)]
		$\{X\in\ind\mathscr{I}\,|\,\chi([X])=\chi([I_i])\}$ is a finite set for any $i\in\{1,\dots,m\}$.
		\item[(f)]
		$\{X\in\ind\mcc^0\,|\,\chi([X])=l\}$ is a finite set for any positive integer $l$.
	\end{enumerate}
\end{mainthmH}
The above theorem shows that an $n$-hereditary algebra $\Lambda$ with the acyclic Gabriel quiver is $n$-representation finite if and only if the number of isomorphism classes of indecomposable $\Lambda$-modules $X$ in $\mcc^0$ with $\chi([X])=1$ is finite. 

Finally, by giving explicit counterexamples, we show that unlike the classical case the converse of Theorem \ref{thm:mainD} does not hold.
	
	The structure of this paper is as follows. In Section \ref{s2}, we prepare some fundamental concepts and known facts that will be used throughout the paper. In section \ref{s3}, we prove $n$-Auslander-Reiten duality for $n$-hereditary algebras. Also, we generalize some well-known properties of hereditary algebras to the $n$-hereditary case and prove Theorem \ref{thm:MainB}. In section \ref{s4}, we give useful  formulas for the dimension vector of the $n$-Auslander--Reiten translations of indecomposable non-projective and non-injective modules in $\mcc^0$ in terms of Coxeter transformation. Section \ref{s5} is devoted to the proof of Theorem \ref{thm:mainB}. We prove Proposition \ref{mainproA} and Theorem \ref{thm:mainC} in section \ref{s6}. Finally, in the last section, we introduce the Grothendieck group $\mathrm{K}_0(\mcc^0)$, for $n$-hereditary algebras. Then, we prove Theorems \ref{thm:mainE}, \ref{thm:mainF} and \ref{thm:mainD}. Furthermore, we show that the converse of Theorem \ref{thm:mainD} does not hold by providing  counterexamples.
\subsection{Notation}
Throughout this paper, we fix a positive integer $n$. We assume that $K$ is a field and we denote by $D$ the $K$-dual $\Hom_K(-,K)$. All
algebras in this paper are indecomposable and finite dimensional $K$-algebra. Also, all modules are right modules and for any indecomposable finite dimensional $K$-algebra $\Lambda$, we denote by $\modd \Lambda$ the abelian category of all finitely generated right $\Lambda$-modules. Also, we denote by $\mathrm{proj}\,\Lambda$ and $\mathrm{inj}\,\Lambda$ the full subcategories of $\modd\Lambda$ consisting of projective and injective objects, respectively. For the bounded derived category of $\modd \Lambda$, we apply the notation $\md_\Lambda$. If $\mathcal{X}$ is a class of objects of an additive category $\mathcal{C}$, then we denote by $\add\mathcal{X}$ the full subcategory of $\mathcal{C}$ whose objects are direct summands of finite direct sums of objects in $\mathcal{X}$. Specially, for an object $M\in\mathcal{C}$, we denote by $\add M$ the subcategory of $\mathcal{C}$ consisting of direct summands of finite direct sums of copies
of $M$. Also, we denote by $\ind\mathcal{C}$ the class of isomorphism classes of indecomposable objects in the Krull--Schmidt category $\mathcal{C}$. For full subcategories $\{\mathcal{X}_1,\dots,\mathcal{X}_t\}$ of $\mathcal{C}$, we denote by $\bigvee_{i=1}^t\mathcal{X}_i$
 the full subcategory $\add(\bigcup_{i=1}^t\ind\mathcal{X}_i)$ of $\mathcal{C}$.

\section{Preliminaries}\label{s2}
In this section, we collect some basic definitions and facts of higher dimensional Auslander--Reiten theory that we need throughout this paper. We refer the reader to papers \cite{HIO,I3,I4,I1,IO} for further details.

Recall that a subcategory $\mathcal{M}$ of $\modd\Lambda$ is called \textit{contravariantly finite} if for any $X\in\modd \Lambda$, there exists a morphism
$f\in\Hom_\Lambda(M,X)$ with $M\in\mathcal{M}$ such that the sequence $\Hom_\Lambda(-,M)
{\overset{f}{\to}}
\Hom_\Lambda(-,X) \to 0$ of abelian groups is exact on $\mathcal{M}$.
Dually a \textit{covariantly finite} subcategory of $\modd\Lambda$ is defined. A contravariantly finite and covariantly finite subcategory of $\modd \Lambda$ is called \textit{functorially finite} subcategory (see \cite[Page 81]{AS1}).

Higher dimensional Auslander--Reiten theory was introduced by Iyama \cite{I1}. One of the main concepts in this theory is $n$-cluster tilting subcategories. A functorially finite full subcategory $\mathcal{M}$ of $\modd\Lambda$ is called \textit{$n$-cluster tilting} if
\begin{align*}
\mathcal{M}&=\{X\in\text{$\modd\Lambda$}\,|\,\mathrm{Ext}^i_\Lambda(X,\mathcal{M})=0, \,\,\, \text{for } 0<i<n\}\\
&=\{X\in\text{$\modd\Lambda$}\,|\,\mathrm{Ext}^i_\Lambda(\mathcal{M},X)=0, \,\,\, \text{for } 0<i<n\}.
\end{align*}
Also, $M\in\modd\Lambda$ is called an \textit{$n$-cluster tilting object} if $\add M$ is an $n$-cluster tilting subcategory of $\modd\Lambda$ (see \cite[Definition 2.2]{I1} and \cite[Definition 1.1]{I4}). 

In the rest of the paper, we assume that $\Lambda$ has the global dimension at most $n$. We denote the \textit{Nakayama functor} and its quasi-inverse on $\md_\Lambda$ by
 \begin{align*}
 \nu&\coloneqq D\boldsymbol{\mathrm{R}}\Hom_\Lambda(-,\Lambda):\md_\Lambda\longrightarrow \md_\Lambda,\\
 \nu^{-1}&\coloneqq \boldsymbol{\mathrm{R}}\Hom_{\Lambda^{\mathrm{op}}}(D-,\Lambda):\md_\Lambda\longrightarrow \md_\Lambda,
 \end{align*}
respectively. We denote by $[1]$ the suspension in $\md_\Lambda$ and we recall two important autoequivalences
 \begin{align*}
 \nu_n&\coloneqq \nu\circ[-n] :\md_\Lambda\longrightarrow \md_\Lambda,\\
 \nu^{-1}_n&\coloneqq \nu^{-1}\circ[n] :\md_\Lambda\longrightarrow \md_\Lambda.
 \end{align*}
Under the assumption $\gd\Lambda\leq n$, we have the following useful property.
\begin{proposition}[\protect{\cite[Proposition 2.3]{HIO}}]\label{thio0}
Assume that $i,j\in\mathbb{Z}$ and $i<j$. Then,
$$\Hom_{\md_\Lambda}(\nu_n^i(\Lambda),\nu_n^j(\Lambda))=0.$$
\end{proposition}
We recall the \textit{$n$-Auslander–Reiten translations}
\begin{align*}
	\tau_n &\coloneqq D\Ext^n_{\Lambda}(-,\Lambda):\modd\Lambda\longrightarrow\modd\Lambda,\\
	\tau_n^{-}&\coloneqq \Ext^n_{\Lambda^{\mathrm{op}}}(D-,\Lambda):\modd\Lambda\longrightarrow\modd\Lambda.
\end{align*}
Note that $\tau_n\cong\mathrm{H}^0(\nu_n-)$ and $\tau_n^{-}\cong\mathrm{H}^0(\nu_n^{-1}-)$,
 where $\mathrm{H}^0$ is $0$-th homology.

$n$-Auslander--Reiten translations play an important role in higher Auslander--Reiten theory. Specially, the structure of $n$-representation finite algebras is controlled by them. 
\begin{definition}[\protect{\cite[Definition 2.2]{IO}}]
$\Lambda$ is called \textit{$n$-representation finite} if $\gd\Lambda\leq n$ and there exists an $n$-cluster tilting object in $\modd\Lambda$.
\end{definition}
We have the following facts of $n$-representation finite algebras.
\begin{proposition}[\protect{\cite[Proposition 1.3 and Theorem 1.6]{I4}}]\label{per}
Let $\Lambda$ be an $n$-representation finite algebra and $M\in\modd\Lambda$ be an $n$-cluster tilting object.
\begin{itemize}
	\item[($\rm{a}$)]
		Assume that $\{P_1,\dots,P_m\}$ and $\{I_1,\dots,I_m\}$ are the complete set of the isomorphism classes of indecomposable projective and indecomposable injective $\Lambda$-modules, respectively. Then there exist a permutation $\sigma$ on $\{1,\dots,m\} $ and non-negative integers $t_i$ such that $\tau_n^{t_i} I_i\cong P_{\sigma(i)}$, for any $i \in \{1,\dots,m\}$.
	\item[($\rm{b}$)]
 	$\add M=\add\{\tau_n^{-i}(\Lambda)\,|\,i\geq 0\}=\add\{\tau_n^{i}(D\Lambda)\,|\,i\geq 0\}$ 
 is the unique $n$-cluster tilting subcategory of $\modd\Lambda$. 
\end{itemize}
\end{proposition}
Note that $1$-representation finite algebras
are exactly representation finite hereditary algebras. In the classical case $n=1$,  representation infinite hereditary algebras can be considered as a natural counterpart to representation finite ones. Motivated by this fact, Herschend, Iyama and Oppermann introduced $n$-representation infinite and $n$-hereditary algebras \cite{HIO}.
\begin{definition}[\protect{\cite[Definitions 2.7, 3.2]{HIO}}]\label{n-h}
\begin{enumerate}
\item
$\Lambda$ is called \textit{$n$-representation infinite} if $\gd\Lambda\leq n$ and any indecomposable projective $\Lambda$-module $P$ satisfies that $\nu_n^{-i}(P)\in \modd\Lambda$, for any $i\geq 0$.
\item
$\Lambda$ is called \textit{$n$-hereditary} if $\gd\Lambda\leq n$ and  $\nu_n^{i}(\Lambda)\in \md^{n\mathbb{Z}}_\Lambda$ for any $i\in\mathbb{Z}$, where
$$\md^{n\mathbb{Z}}_\Lambda\coloneqq\{X\in \md_\Lambda\,|\,\mathrm{H}^i(X)=0,\,\,\,  i\in\mathbb{Z}\backslash n\mathbb{Z}\}.$$
\end{enumerate}
\end{definition}
In \cite{HIO}, the authors proved a dichotomy of $n$-hereditary algebras based on the classes of $n$-representation finite and $n$-representation infinite algebras.

\begin{theorem}[\protect{\cite[Theorem 3.4]{HIO}}]\label{de}
$\Lambda$ is $n$-hereditary if and only if it is either $n$-representation finite or $n$-representation infinite.
\end{theorem}
Let $\Lambda$ be an $n$-hereditary algebra. Full subcategories
\begin{equation}
\uu\coloneqq \add\{\nu_n^i(\Lambda)\,|\,i\in\mathbb{Z}\}\notag
\end{equation}
and
\begin{align*}
\mcc \coloneqq &\{X\in \md_\Lambda\,|\,\Hom_{\md_\Lambda}(\uu,X[i])=0,\,\,\,  i\in\mathbb{Z}\backslash n\mathbb{Z}\}\\
=&\{X\in \md_\Lambda\,|\,\Hom_{\md_\Lambda}(X,\uu[i])=0,\,\,\,  i\in\mathbb{Z}\backslash n\mathbb{Z}\}
\end{align*}
of $\md_\Lambda$ and the full subcategory
\begin{equation}
\mcc ^0\coloneqq (\modd\Lambda)\cap\mcc\notag
\end{equation}
of $\modd\Lambda$ play an important role in the higher dimensional Auslander--Reiten theory (see \cite[Definition 1.22]{I4} and \cite[Section 4]{HIO}). It is obvious that for $n=1$, $\mcc$ and $\mcc^0$ are exactly $\md_\Lambda$ and $\modd\Lambda$, respectively. 

We also recall important full subcategories
\begin{align*}
\mathscr{P}&\coloneqq \add\{\tau_n^{-i}(\Lambda)\,|\,i\geq 0\},\\
\mathscr{I}&\coloneqq \add\{\tau_n^{i}(D\Lambda)\,|\,i\geq 0\},
\end{align*}
of $\modd\Lambda$, which have already appeared in the study of $n$-hereditary algebras (for example, see \cite{HIO,I4,IO,IO1}). Modules in $\mathscr{P}$ and $\mathscr{I}$ are called \textit{$n$-preprojective} and \textit{$n$-preinjective}  $\Lambda$-modules, respectively (see \cite[Definition 4.7]{HIO}).
Note that one can be considered the subcategories $\mathscr{P}$ and $\mathscr{I}$  as higher dimensional analogues of the classical preprojective and preinjective components of hereditary algebras, respectively.
\begin{theorem}[\protect{\cite[Theorems 1.23 and 1.6]{I4} and \cite[Remark 4.6]{HIO}}]\label{th1}
	Let $\Lambda$ be an $n$-representation finite algebra. Then $\uu$ is an $n$-cluster tilting subcategory of $\md_\Lambda$. Moreover, $\mcc=\uu$ and $\mcc ^0=\mathscr{P}=\mathscr{I}$ is the unique $n$-cluster tilting subcategory of $\modd\Lambda$.
\end{theorem}
In $n$-representation infinite case, by the definition, we have
\begin{align*}
\mathscr{P}&\coloneqq\add\{\tau_n^{-i}(\Lambda)\,|\,i\geq 0\}= \add\{\nu_n^{-i}(\Lambda)\,|\,i\geq 0\},\\
\mathscr{I}&\coloneqq\add\{\tau_n^{i}(D\Lambda)\,|\,i\geq 0\}= \add\{\nu_n^{i}(D\Lambda)\,|\,i\geq 0\}.
\end{align*}
Moreover, in this case, the modules in
the subcategory 
\begin{align*}
\mathscr{R}&\coloneqq \{X\in \modd\Lambda\,|\,\Ext_{\Lambda}^i(\mathscr{P},X)=0=\Ext_{\Lambda}^i(X,\mathscr{I}),\,\,\, i>0\}
\end{align*}
of $\modd\Lambda$ is called \textit{$n$-regular} (see \cite[Definition 4.14]{HIO}). For the case $n=1$, $\mathscr{R}$ is exactly the class of regular modules for representation infinite hereditary algebras. We recall some key properties of $n$-representation infinite algebras.
\begin{proposition}[\protect{\cite[Proposition 4.10 and Theorem 4.18]{HIO}}]\label{Thio}
Let $\Lambda$ be an $n$-representation infinite algebra. Then the following hold.
\begin{itemize}
\item[(a)]
There is a bijection $\ind(\mathrm{proj}\,\Lambda)\times \mathbb{Z}_{\geq0}\to \ind\mathscr{P}$ given by $(P,i)\mapsto \nu_n^{-i}(P)$.
\item[(b)]
There is a bijection $\ind(\mathrm{inj}\,\Lambda)\times \mathbb{Z}_{\geq0}\to \ind\mathscr{I}$ given by $(I,i)\mapsto \nu_n^{i}(I)$.
\item[(c)]
$\uu=\mathscr{I}[-n]\vee\mathscr{P}$.
\item[(d)]
$\Hom_\Lambda(\mathscr{I},\mathscr{P})=0$ and $\mathscr{P}\cap\mathscr{I}=0$.
\item[(e)]
$\Hom_\Lambda(\mathscr{R},\mathscr{P})=0$ and $\Hom_\Lambda(\mathscr{I},\mathscr{R})=0$.
\item[(f)]
$\mcc^0=\mathscr{P}\vee\mathscr{R}\vee\mathscr{I}$.
\end{itemize}
\end{proposition}
Assume that $\mcc^0_P$ and $\mcc^0_I$ are the full subcategories of $\mcc^0$ consisting of modules without non-zero projective and injective direct summands, respectively. Set $\mathscr{P}_P\coloneqq\mathscr{P}\cap \mcc^0_P$ and $\mathscr{I}_I\coloneqq\mathscr{I}\cap \mcc^0_I$, then we have the following fact.
 \begin{proposition}[\protect{\cite[Proposition 2.20]{IO1} and \cite[Proposition 4.21]{HIO}}]\label{tunu}
Let $\Lambda$ be an $n$-hereditary algebra. There exist mutually quasi-inverse equivalences
 $$\xymatrix{
\mcc^0_P \ar@<.7ex>[rr]^{\nu_n=\tau_n} &&\mcc^0_I. \ar@<.7ex>[ll]^{\nu_n^{-1}=\tau_n^-}}$$
Moreover, the above equivalences restrict to the following equivalences for the  $n$-representation infinite case.
\begin{align*}
\xymatrix{
\mathscr{P}_P \ar@<.7ex>[rr]^{\nu_n=\tau_n} &&\mathscr{P} \ar@<.7ex>[ll]^{\nu_n^{-1}=\tau_n^-}},\quad\quad\,\,\
\xymatrix{
\mathscr{R} \ar@<.7ex>[rr]^{\nu_n=\tau_n} &&\mathscr{R} \ar@<.7ex>[ll]^{\nu_n^{-1}=\tau_n^-}},\,\,\quad\quad
\xymatrix{\mathscr{I} \ar@<.7ex>[rr]^{\nu_n=\tau_n} &&\mathscr{I}_I. \ar@<.7ex>[ll]^{\nu_n^{-1}=\tau_n^-}}
\end{align*}
\end{proposition}
Let $\Lambda$ be an $n$-representation infinite algebra. We recall that in the classical case $n=1$, the subcategory $\mathscr{R}$ is non-zero (for example see \cite[Proposition $\rm{VIII}.2.9$]{ASS}). For the case $n\geq 2$, this remains a conjecture as posed by Herschend et al. 
\begin{conjecture}[\protect{\cite[Conjecture 6.9]{HIO}}]\label{r0}
$\mathscr{R}$ is non-zero for any $n$-representation infinite algebra.
\end{conjecture}

Herschend et al. in \cite{HIO} proved that the conjecture is true for $n$-representation tame algebras. To recall the definition of $n$-representation tame algebras, first we recall the concept of $(n+1)$-preprojective algebras.

Let $\Lambda$ be an $n$-representation infinite algebra. The \textit{$(n+1)$-preprojective algebra} $\hat{\Lambda}$ of $\Lambda$ is the tensor algebra of the $\Lambda$-bimodule $\Ext^n
_\Lambda(D\Lambda, \Lambda)$ over $\Lambda$:
$$\hat{\Lambda}\coloneqq T_\Lambda\Ext^n
_\Lambda(D\Lambda, \Lambda),$$
(see \cite[Section 4.2]{HIO} and \cite[Section 4]{K}).
We recall that a ring $\Gamma$ is called \textit{Noetherian} $R$-algebra if $R$ is a commutative Noetherian ring and $\Gamma$ is a finitely generated $R$-module.
\begin{definition}[\protect{\cite[Definition 6.10]{HIO}}]
An $n$-representation infinite algebra $\Lambda$ is called \textit{$n$-representation tame} if its $(n+1)$-preprojective algebra is a Noetherian algebra. 
\end{definition}
We recall some important properties of $n$-representation tame algebras.
\begin{proposition}[\protect{\cite[Propositions 6.13 and 6.23]{HIO}}]\label{tame}
 Let $\Lambda$ be an $n$-representation tame algebra. Then the following hold.
 \begin{itemize}
 \item[(a)]
 The category $\mathscr{R}$ is non-zero.
 \item[(b)]
There exists a positive integer $t$ such that $\tau_n^t(X)\cong X$ holds for any $X\in\mathscr{R}$.
 \end{itemize}
\end{proposition}
\section{$n$-Auslander-Reiten duality for $n$-hereditary algebras}\label{s3}
In this section, we first recall $n$-Auslander-Reiten duality for $n$-cluster tilting subcategories of $\modd\Lambda$. Then we show that one can be generalized this duality to subcategories $\mathscr{P}$ and $\mathscr{I}$ of $\mcc^0$ for $n$-hereditary algebras.
Moreover, we generalize some well-known properties of hereditary algebras to $n$-hereditary algebras. Finally, we provide a criterion for determining the $n$-representation infiniteness of $n$-hereditary algebras. Our criterion is also a sufficient condition for the validity of Conjecture \ref{r0}.

We recall that a subcategory $\mcm$ of $\modd\Lambda$ is called \textit{$n$-rigid} if $\Ext^i
_\Lambda(\mcm,\mcm) = 0$, for each $i\in\{1,\dots,n-1\}$. Let $\Lambda$ be an $n$-representation finite algebra. It is obvious that $\mcc^0$ is an $n$-rigid subcategory of $\modd\Lambda$. We generalize this fact to $n$-hereditary algebras as follows.

 \begin{proposition}\label{pic}
Let $n\geq 2$ and $\Lambda$ be an $n$-hereditary algebra. Then, $\Ext^i
_\Lambda(\mcc^0,\mathscr{P}\vee\mathscr{I})=\Ext^i
_\Lambda(\mathscr{P}\vee\mathscr{I},\mcc^0)=0$ for any $i\in\{1,\dots,n-1\}$.
 \end{proposition}
 \begin{proof}
 By the definition of $\mcc^0$, we have
 $$\Hom_{\md_\Lambda}(\uu,\mcc^0[i])=\Hom_{\md_\Lambda}(\mcc^0,\uu[i])=0,\,\forall i\in\mathbb{Z}\backslash n\mathbb{Z}.$$
It is obvious that $\mathscr{P}\subseteq\uu$, so we have
 $$\Hom_{\md_\Lambda}(\mathscr{P},\mcc^0[i])=\Hom_{\md_\Lambda}(\mcc^0,\mathscr{P}[i])=0,\,\forall i\in\mathbb{Z}\backslash n\mathbb{Z}.$$
Specially, we have
\begin{align*}
\Ext^i
_\Lambda(\mathscr{P},\mcc^0)=\Hom_{\md_\Lambda}(\mathscr{P},\mcc^0[i])=0,\\
\Ext^i
_\Lambda(\mcc^0,\mathscr{P})=\Hom_{\md_\Lambda}(\mcc^0,\mathscr{P}[i])=0,
\end{align*}
for any $i\in\{1,\dots,n-1\}$. If $\Lambda$ is $n$-representation finite, then it is obvious that $\mathscr{I}\subseteq\uu$ and so similarly, we have
\begin{align*}
\Ext^i
_\Lambda(\mathscr{I},\mcc^0)=\Hom_{\md_\Lambda}(\mathscr{I},\mcc^0[i])=0,\\
\Ext^i
_\Lambda(\mcc^0,\mathscr{I})=\Hom_{\md_\Lambda}(\mcc^0,\mathscr{I}[i])=0,
\end{align*}
for any $i\in\{1,\dots,n-1\}$. If $\Lambda$ is $n$-representation infinite then by Proposition \ref{Thio}$(\mathrm{c})$, $\mathscr{I}[-n]\subseteq\uu$ and so we have
\begin{equation}\label{eqq0}
\Hom_{\md_\Lambda}(\mathscr{I}[-n],\mcc^0[i])=\Hom_{\md_\Lambda}(\mathscr{I},\mcc^0[i+n])=0,
\end{equation}
\begin{equation}\label{eqq1}
\Hom_{\md_\Lambda}(\mcc^0,\mathscr{I}[i-n])=0,\qquad\qquad\quad\quad\qquad\qquad\,
\end{equation}
for any $i\in\mathbb{Z}\backslash n\mathbb{Z}$.
Put $j=i+n$. By \eqref{eqq0} for any $i\in\{-1,\dots,-n+1\}$,  we have
\begin{align*}
\Ext^j
_\Lambda(\mathscr{I},\mcc^0)=\Hom_{\md_\Lambda}(\mathscr{I},\mcc^0[j])=0,\,\forall j\in\{1,\dots, n-1\}.
\end{align*}
Put $j=i-n$. By \eqref{eqq1} for any $i\in\{1+n,\dots,2n-1\}$, we have
\begin{align*}
\Ext^j
_\Lambda(\mcc^0,\mathscr{I})=\Hom_{\md_\Lambda}(\mcc^0,\mathscr{I}[j])=0,\,\forall j\in\{1,\dots, n-1\},
\end{align*}
and the proof is complete.
\end{proof}
Proposition \ref{pic} is not only frequently used in the rest of the paper but is also of independent interest. Iyama has given $n$-Auslander-Reiten duality for $n$-cluster tilting subcategories which is usual Auslander-Reiten duality for $n=1$.
Let $\mcm$ be an $n$-cluster tilting subcategory of $\modd\Lambda$. Then, there exist the following
 functorial isomorphisms for any $X, Y\in\mcm$ and $Z\in\modd\Lambda$.
$$
\underline{\Hom}_\Lambda(X,Z)\cong D\Ext^n_\Lambda(Z,\tau_nX)\quad\,\,\text{and}\quad\,\,
\overline{\Hom}_\Lambda(Z,Y)\cong D\Ext^n_\Lambda(\tau_n^-Y,Z),$$ 
(see \cite[Theorem 2.3.1]{I1}).
To prove these isomorphisms, Iyama applied the following more general result.
\begin{theorem}[\protect{\cite[Theorem 1.5]{I1}}]\label{gARd}
Consider the following subcategories of $\modd\Lambda$.
\begin{align*}
\mathscr{X}&\coloneqq\{X\in\modd\Lambda\,|\,\mathrm{Ext}^i_\Lambda(X,\Lambda)=0, \,\,\,  0<i<n\},\\
\mathscr{Y}&\coloneqq\{X\in\modd\Lambda\,|\,\mathrm{Ext}^i_\Lambda(D\Lambda,X)=0, \,\,\,  0<i<n\}.
\end{align*}
Then, there exist the following functorial isomorphisms for any $X\in\mathscr{X}$, $Y\in\mathscr{Y}$ and $Z\in\modd\Lambda$.
$$
\underline{\Hom}_\Lambda(X,Z)\cong D\Ext^n_\Lambda(Z,\tau_nX)\quad\,\,\text{and}\quad\,\,
\overline{\Hom}_\Lambda(Z,Y)\cong D\Ext^n_\Lambda(\tau_n^-Y,Z).$$
\end{theorem}
By using the above theorem, we can generalize $n$-Auslander-Reiten duality to subcategory $\mcc^0$ of $n$-hereditary algebras.
\begin{theorem}\label{ARd1}
Let $\Lambda$ be an $n$-hereditary algebra. Then, there exist the following
 functorial isomorphisms for any $X, Y\in\mcc^0$ and $Z\in\modd\Lambda$.
$$
\underline{\Hom}_\Lambda(X,Z)\cong D\Ext^n_\Lambda(Z,\tau_nX)\quad\,\,\text{and}\quad\,\,
\overline{\Hom}_\Lambda(Z,Y)\cong D\Ext^n_\Lambda(\tau_n^-Y,Z).$$
\end{theorem}
\begin{proof}
By Theorem \ref{gARd}, it is enough to show that $\mcc^0\subseteq \mathscr{X}$ and  $\mcc^0\subseteq\mathscr{Y}$. By definitions of $\mathscr{P}$ and $\mathscr{I}$, it is obvious that $\Lambda\in\mathscr{P}$ and $D\Lambda\in\mathscr{I}$. Assume that $X\in\mcc^0$. Proposition \ref{pic} implies that $\Ext_\Lambda^i(X,\Lambda)=\Ext_\Lambda^i(D\Lambda,X)=0$, for any $i\in\{1,\dots,n-1\}$. Therefore, $X\in\mathscr{X}$ and $X\in\mathscr{Y}$ and the result follows.
\end{proof}
Now, we give the following properties of $n$-hereditary algebras, which is well-known for hereditary algebras (for example see \cite[Section $\rm{VII}.6$]{SY}).
\begin{theorem}\label{n-ext}
Let $\Lambda$ be an $n$-hereditary algebra.
\begin{itemize}
\item[(a)]
Assume that $X\in\ind\mathscr{P}$.
\begin{itemize}
\item[(i)]
There are a unique integer $i\geq 0$ and a unique indecomposable projective $\Lambda$-module $P$ such that $X\cong\nu_n^{-i}P\cong\tau_n^{-i}P$.
\item[(ii)]
$\dim_K\End_\Lambda(X)=\dim_K\End_\Lambda(P)$.
\item[(iii)]
There exists an integer $t\geq0$ such that $\tau_n^tX\neq0$ and $\tau_n^{t+1}X=0$.
\item[(iv)]
$\Ext_\Lambda^i(X,X)=0$, for any $i\in\{1,\dots,n\}$.
\end{itemize}
\item[(b)]
Assume that $X\in\ind\mathscr{I}$.
\begin{itemize}
\item[(i)]
There are a unique integer $i\geq 0$ and a unique indecomposable injective $\Lambda$-module $I$ such that $X\cong\nu_n^{i}I\cong\tau_n^{i}I$.
\item[(ii)]
$\dim_K\End_\Lambda(X)=\dim_K\End_\Lambda(I)$.
\item[(iii)]
There exists an integer $r\geq0$ such that $\tau_n^{-r}X\neq0$ and $\tau_n^{-r-1}X=0$.
\item[(iv)]
$\Ext_\Lambda^i(X,X)=0$, for any $i\in\{1,\dots,n\}$.
\end{itemize}
\end{itemize}
\end{theorem}
\begin{proof}
We only prove $(\mathrm{a})$, the proof of $(\mathrm{b})$ is similar.

$(\mathrm{i})$ If $\Lambda$ is $n$-representation infinite, then the result follows from the definition of $\mathscr{P}$ and Proposition \ref{Thio}$(\mathrm{a})$. If $\Lambda$ is $n$-representation finite, then by Theorem \ref{th1}, $\mathscr{P}=\mcc ^0$ and the result follows from Propositions \ref{per}$(\mathrm{b})$ and \ref{tunu}.

$(\mathrm{ii})$ By the part $(\mathrm{i})$, there are an integer $i\geq 0$ and an indecomposable projective $\Lambda$-module $P$ such that $X\cong\nu_n^{-i}P\cong\tau_n^{-i}P$. Since $\nu_n^{-i}$ is an equivalence, we have
\begin{align*}
\dim_K\End_\Lambda(X) &=\dim_K\Hom_\Lambda(X,X)\\
&=\dim_K\Hom_\Lambda(\tau_n^{-i}P,\tau_n^{-i}P)\\
&=\dim_K\Hom_{\md_\Lambda}(\tau_n^{-i}P,\tau_n^{-i}P)\\
&=\dim_K\Hom_{\md_\Lambda}(\nu_n^{-i}P,\nu_n^{-i}P)\\
&=\dim_K\Hom_{\md_\Lambda}(P,P)\\
&=\dim_K\Hom_{\Lambda}(P,P)\\
&=\dim_K\End_\Lambda(P).
\end{align*}

$(\mathrm{iii})$ If $X$ is projective, then for $t=0$ the result follows. If $X$ is non-projective, then by the part $(\mathrm{i})$, there are an integer $t> 0$ and an indecomposable projective $\Lambda$-module $P$ such that $X\cong\tau_n^{-t}P$. By Proposition \ref{tunu}, $\tau_n^{t}X\cong\tau_n^{t}\tau_n^{-t}P\cong P\neq0$ and $\tau_n^{t+1}X\cong\tau_nP=0$.

$(\mathrm{iv})$
Proposition \ref{pic} implies that $\Ext_\Lambda^i(X,X)=0$, for any $i\in\{1,\dots,n-1\}$. Therefore, it is enough to show that $\Ext_\Lambda^n(X,X)=0$. By the part $(\mathrm{i})$, there are an integer $i\geq 0$ and an indecomposable projective $\Lambda$-module $P$ such that $X\cong\nu_n^{-i}P\cong\tau_n^{-i}P$. By applying $n$-Auslander-Reiten duality, Theorem \ref{ARd1}, we have 
\begin{align}\label{n1}
	\Ext_\Lambda^n(X,X) \cong\Ext_\Lambda^n(\tau_n^{-i}P,\tau_n^{-i}P)\cong D\underline{\Hom}_\Lambda(\tau_n^{-i-1}P,\tau_n^{-i}P).
\end{align}
If $\Lambda$ is $n$-representation finite and $X$ is injective, then $\tau_n^{-i-1}P=0$ and so by \eqref{n1}, $	\Ext_\Lambda^n(X,X)=0$. Consider the remaining two cases,  either $\Lambda$ is $n$-representation infinite or $\Lambda$ is $n$-representation finite and $X$ is non-injective.
Proposition \ref{tunu} implies that $\nu_n^{-i-1}P\cong\tau_n^{-i-1}P$ and therefore,
$$\Hom_\Lambda(\tau_n^{-i-1}P,\tau_n^{-i}P)\cong\Hom_{\md_\Lambda}(\nu_n^{-i-1}P,\nu_n^{-i}P).$$
Since $-i-1<-i$, Proposition \ref{thio0} implies that  $\Hom_{\md_\Lambda}(\nu_n^{-i-1}P,\nu_n^{-i}P)=0$ and so $\Hom_\Lambda(\tau_n^{-i-1}P,\tau_n^{-i}P)=0$. Therefore by \eqref{n1}, $\Ext_\Lambda^n(X,X)=0$. 
\end{proof}
	We recall that $X\in\modd \Lambda$ is called a \textit{brick} if $\End_\Lambda(X)$ is a division $K$-algebra. In the case that the field $K$ is algebraically closed, every finite dimensional division $K$-algebra is isomorphic to $K$. If there are only finitely many isomorphism classes of
bricks in $\modd \Lambda$, then $\Lambda$ is called \textit{brick--finite}. Demonet, Iyama and Jasso have studied these algebras in \cite{DIJ}. They have given equivalent conditions to brick--finiteness of $\Lambda$. It is obvious that if $\Lambda$ is representation finite, then it is also brick--finite. But, in general the converse is not true (for example, see \cite{Mi} and \cite{Ma}). For a hereditary $K$-algebra 
$\Lambda$ over algebraically closed field $K$, it is known that $\Lambda$ is representation finite if and only if $\Lambda$ is brick--finite (see \cite[Section $\rm{VII}.6$]{SY}). In the following  theorem, we prove higher dimensional version of this fact.
\begin{theorem}\label{bricks}
	Let $K$ be an algebraically closed field. Assume that $\Lambda$ is an $n$-hereditary $K$-algebra with the acyclic Gabriel quiver.
 Then the following are equivalent.
	\begin{enumerate}
		\item[(a)]
		$\Lambda$ is $n$-representation finite.
		\item[(b)]
		There are finitely many isomorphism classes of bricks in $\mathscr{P}$.
		\item[(c)]
		There are finitely many isomorphism classes of bricks in $\mathscr{I}$.
		\item[(d)]
		There are finitely many isomorphism classes of bricks in $\mcc^0$. 
	\end{enumerate}
\end{theorem}
\begin{proof}
	We only prove the equivalence of $(\mathrm{a})$ and $(\mathrm{b})$. The proof of other parts is similar. Clearly, $(\mathrm{a})$ implies $(\mathrm{b})$. Assume that there are finitely many isomorphism classes of bricks in $\mathscr{P}$. Suppose to the contrary that $\Lambda$ is not $n$-representation finite. By Theorem \ref{de}, $\Lambda$ is $n$-representation infinite. Consider an indecomposable projective $\Lambda$-module $P$. By Proposition \ref{Thio}$(\mathrm{a})$, $\tau_n^{-i}P$ are non-isomorphic for any $i\geq0$. Since the Gabriel quiver of $\Lambda$ is acyclic,  $\dim_K\End_\Lambda(P)=1$. Theorem \ref{n-ext}$(\mathrm{a})$, implies that $\dim_K\End_\Lambda(\tau^{-i}P)=\dim_K\End_\Lambda(P)=1$, for any $i\geq0$. Therefore, there are infinitely many isomorphism classes of bricks in $\mathscr{P}$ and this is a contradiction.
\end{proof}
As a consequence of Theorem \ref{n-ext}, we give a criterion to describe whether an $n$-hereditary algebra is $n$-representation infinite. This criterion is also a sufficient condition for 
the validity of Conjecture \ref{r0}.
\begin{corollary}\label{cor0}
Let $\Lambda$ be an $n$-hereditary algebra.
\begin{itemize}
\item[(a)]
If there are $X\in\ind\mcc^0$ and $i\in\{1,\dots,n\}$ satisfying $\Ext_\Lambda^i(X,X)\neq 0$, then $\Lambda$ is $n$-representation infinite. Moreover, in this case $\mathscr{R}\neq 0$.                   
\item[(b)]
Assume that $n\geq 2$. If there are $X\in\mcc^0$ and $i\in\{1,\dots,n-1\}$ satisfying $\Ext_\Lambda^i(X,X)\neq 0$, then $\Lambda$ is $n$-representation infinite. Moreover, in this case $\mathscr{R}\neq 0$.  
\end{itemize}
\end{corollary}
\begin{proof}
$(\mathrm{a})$ Assume that $\Lambda$ is not $n$-representation infinite. Then by Theorem \ref{de}, $\Lambda$ is $n$-representation finite. By Theorem \ref{th1}, $\mathscr{P}=\mcc ^0$. Therefore, Theorem \ref{n-ext}$(\mathrm{a})$ implies that $\Ext_\Lambda^i(Y,Y)=0$, for any $Y\in\ind\mcc^0$ and $i\in\{1,\dots,n\}$. But this is a contradiction and so $\Lambda$ is $n$-representation infinite. By Proposition \ref{Thio}$(\mathrm{f})$, $\mcc^0=\mathscr{P}\vee\mathscr{R}\vee\mathscr{I}$. We claim that  $X\in\ind\mathscr{R}$ and so $\mathscr{R}\neq 0$.
Assume that $X\in\mathscr{P}\vee\mathscr{I}$. By Theorem \ref{n-ext}, $\Ext_\Lambda^i(Y,Y)=0$ for any $Y\in\ind(\mathscr{P}\vee\mathscr{I})$ and $i\in\{1,\dots,n\}$. So $X\in\mathscr{R}$ and the result follows.

$(\mathrm{b})$ Assume that $\Lambda$ is not $n$-representation infinite. Then by Theorem \ref{de}, $\Lambda$ is $n$-representation finite. Theorem \ref{th1} implies that  $\mcc ^0$ is an $n$-cluster tilting subcategory of $\modd\Lambda$. Therefore $\Ext_\Lambda^i(X,X)=0$, for any $X\in\mcc^0$ and $i\in\{1,\dots,n-1\}$. But this is a contradiction and so $\Lambda$ is $n$-representation infinite. Consider the decomposition $X=\bigoplus_{i=1}^sX_i$ of $X$ into indecomposable summands. By Proposition \ref{Thio}$(\mathrm{f})$, $\mcc^0=\mathscr{P}\vee\mathscr{R}\vee\mathscr{I}$. We claim that there is an $i\in\{1,\dots,s\}$ such that $X_i\in\mathscr{R}$ and so $\mathscr{R}\neq 0$.
Assume that $X_i\in\mathscr{P}\vee\mathscr{I}$, for any $i\in\{1,\dots,s\}$. Then by Proposition \ref{pic}, we have $\Ext_\Lambda^k(X_i,X_j)=0$, for any $i,j\in\{1,\dots,s\}$ and $k\in\{1,\dots,n-1\}$. Hence, $\Ext_\Lambda^k(X,X)=0$, for any $k\in\{1,\dots,n-1\}$ and this is a contradiction.
\end{proof}
Let $\Lambda$ be a hereditary algebra. Then the converse of Corollary \ref{cor0}$(\mathrm{a})$ is satisfied. In fact, there exists an indecomposable $\Lambda$-module $X$ with $\Ext_\Lambda^1(X,X)\neq0$ if and only if  $\Lambda$ is representation infinite (see \cite[Proposition $\mathrm{VIII}.2.9$]{ASS} or \cite[Proposition $\mathrm{VIII}.3.4$]{ARS}). Motivated by this result, we pose the following question.
\begin{question}
Let $\Lambda$ be an $n$-representation infinite algebra.
Are there $X\in\ind\mcc^0$ and $i\in\{1,\dots,n\}$ satisfying $\Ext_\Lambda^i(X,X)\neq 0$.
\end{question}
Note that if the answer of the above question is positive, then Theorem \ref{n-ext} implies that Conjecture \ref{r0} is true.
\section{Coxeter transformation of $n$-hereditary algebras}\label{s4}
In this section, we recall the concept of Coxeter transformation of $n$-hereditary algebras. Then we give useful formulas for the dimension vector of the $n$-Auslander--Reiten translations of non-projective and non-injective modules in $\ind\mcc^0$ in terms of Coxeter transformation.  

In the rest of the paper, we assume that $K$ is an algebraically closed field and $\Lambda$ is a basic finite dimensional $K$-algebra with a complete set of primitive orthogonal idempotents $\{e_1,\dots,e_m\}$. Also we denote by $\{S_1,\dots,S_m\}$, $\{P_1,\dots,P_m\}$ and $\{I_1,\dots,I_m\}$ the complete set of the isomorphism classes of simple, indecomposable projective and indecomposable injective  $\Lambda$-modules, respectively. The \textit{Cartan matrix} of $\Lambda$ is defined as the $m\times m$ matrix
$$C_\Lambda\coloneqq\left[ \begin{array}{lll}\dim_K e_1\Lambda e_1 &\dots & \dim_K e_1\Lambda e_m\\\quad\quad\vdots&\ddots&\quad\quad\vdots\\\dim_K e_m\Lambda e_1&\ldots&\dim_K e_m\Lambda e_m \end{array}\right]\in\mathrm{Mat}_m(\mathbb{Z}).$$
If $\gd\Lambda<\infty$, then we have the following fact.
\begin{proposition}[\protect{\cite[Proposition $\mathrm{III}$.3.10)]{ASS}}]\label{car}
Let $\gd\Lambda<\infty$. Then $C_\Lambda$ is invertible in the matrix ring $\mathrm{Mat}_m(\mathbb{Z})$.
\end{proposition}
We recall that for $X\in\modd\Lambda$, the \textit{dimension vector} of $X$ is defined as the following $m$-vector.
$$\underline{\dim}X\coloneqq\left[ \begin{array}{l}\dim_K Xe_1\\\quad\quad\vdots\\\dim_K Xe_m \end{array}\right]$$

It is obvious that the $i$-th column of $C_\Lambda$ is $\underline{\dim}P_i$.

For a basic finite dimensional $K$-algebra $\Lambda$ with $\gd\Lambda<\infty$, the \textit{Coxeter matrix} of $\Lambda$ is defined as the matrix
$$\Phi\coloneqq -C_\Lambda^t C_\Lambda^{-1}.$$
Also, the \textit{Coxeter transformation} of $\Lambda$ is defined as the group isomorphism
\begin{align*}
 \Phi:\mathbb{Z}^m &\to\mathbb{Z}^m.\\
 x&\mapsto  \Phi\cdot x
 \end{align*}
 Let $\mathcal{C}$ be an essentially small additive category. We denote  the free abelian group with basis the isomorphism classes of all objects $C$ in $\mathcal{C}$ by $G(\mathcal{C})$. Then the \textit{split Grothendieck group} of $\mathcal{C}$ is defined as
 \begin{equation}
 	\mathrm{K}_0(\mathcal{C},0)\coloneqq G(\mathcal{C})/\langle[A\oplus B]-[A]-[B]\ | \ A,B\in \mathcal{C}\rangle. \notag
 \end{equation}
 For any $X\in\mathcal{C}$, we denote by $[X]_0$ the corresponding element in $\mathrm{K}_0(\mathcal{C},0)$.
 Let $\mathcal{C}$ be an abelian category. The \textit{Grothendieck group} of $\mathcal{C}$ is a quotient group of $\mathrm{K}_0(\mathcal{C},0)$ defined as
 \begin{equation}
 	\mathrm{K}_0(\mathcal{C})\coloneqq\mathrm{K}_0(\mathcal{C},0)/\langle[X]_0-[Y]_0+[Z]_0\,|\,0\to X\to Y\to Z\to0\text{ is a short exact sequence in\,\,} \mathcal{C}\rangle.\notag
 \end{equation}
Now, consider the Grothendieck group $\mathrm{K}_0(\modd\Lambda)$. For any $X\in\modd\Lambda$, we denote by $[X]$ the corresponding element in $\mathrm{K}_0(\modd\Lambda)$. By \cite[Theorem $\mathrm{III}$.3.5]{ASS}, $\{[S_1],\dots,[S_m]\}$ is a basis for $\mathrm{K}_0(\modd\Lambda)$ and $\mathrm{K}_0(\modd\Lambda)\cong\mathbb{Z}^m$. In fact, there is the unique group
isomorphism
\begin{align*}
\mathrm{K}_0(\modd\Lambda) &\to\mathbb{Z}^m\\
	[X] &\mapsto \underline{\dim}X.
\end{align*}‌ 
If $\Lambda$ is of finite global dimension, then the sets $\{[P_1],\dots,[P_m]\}$ and $\{[I_1],\dots,[I_m]\}$ are both bases for
$\mathrm{K}_0(\modd\Lambda)$ (see \cite[Theorem $\mathrm{VIII}$.2.1]{ARS}). Under this setting, one can be define the Coxeter transformation as the group isomorphism
 \begin{align*}
c:\mathrm{K}_0(\modd\Lambda)&\to\mathrm{K}_0(\modd\Lambda).\\
 [P_i]&\mapsto -[I_i]
\end{align*}
It is obvious that $c$ coincides with the Coxeter transformation $\Phi$. In fact, $\Phi(\underline{\dim}P_i)=-\underline{\dim}I_i$, for any $i\in\{1,\dots,m\}$ (see \cite[Lemma $\mathrm{III}$.3.16]{ASS}). For an $n$-hereditary algebra $\Lambda$, the Coxeter matrix and transformation are defined different from the usual cases. In this case due to technical reasons, the multiple $(-1)^n$ is applied instead of the multiple $-1$. Therefore, if $\Lambda$ is an $n$-hereditary algebra then the Coxeter matrix and transformation of $\Lambda$ are defined as follows.
$$\Phi\coloneqq (-1)^nC_\Lambda^t C_\Lambda^{-1}\quad\text{and}\quad c([P_i])\coloneqq (-1)^n[I_i],$$
(for $n$-representation finite and $n$-representation infinite cases, see \cite[Definition 2.1]{M} and \cite[Observation 4.12]{HIO}, respectively). Note that $\Phi$ gives the action of $\nu_n$ on $\mathrm{K}_0(\md_\Lambda)$, where $\mathrm{K}_0(\md_\Lambda)$ is  the Grothendieck group of $\md_\Lambda$ (see \cite[Section $\mathrm{III}.1$]{Ha}). Similar to the usual case, we have the following easy lemma.
\begin{lemma}\label{pi}
Let $\Lambda$ be an $n$-hereditary algebra with the Coxeter matrix  $\,\Phi$. Then $\Phi(\underline{\dim}P_i)=(-1)^n\underline{\dim}I_i$, for any $i\in\{1,\dots,m\}$.	
\end{lemma}
Motivated by Auslander and Platzeck's results in \cite[Proposition 2.2]{AP}, we have the following facts.
\begin{proposition}\label{pro04}
Let $\Lambda$ be an $n$-hereditary algebra. Then the following hold.
\begin{itemize}
\item[(a)]
If $X\in\modd\Lambda$, then $c([X])= \sum_{i=0}^n(-1)^{i+n}[D\Ext_\Lambda ^i(X, \Lambda)]$. 
\item[(b)] 
If $X\in\mcc^0$, then $c([X])=[\tau_nX]+(-1)^n[D\Hom_\Lambda(X, \Lambda)]$.
\item[(c)] 
If $X\in\mcc^0$ is without projective direct summands, then $c([X])=[\tau_nX]$.
\item[(d)] 
If $X\in\mcc^0$ is projective, then $c([X])=(-1)^n[D\Hom_\Lambda(X, \Lambda)]$.
\item[(e)]
If $X\in\modd\Lambda$, then $c^{-1}([X]) = \sum_{i=0}^n(-1)^{i+n}[ \Ext_{\Lambda^{\mathrm{op}}}^i(DX,\Lambda)]$. 
\item[(f)] 
If $X\in\mcc^0$, then $c^{-1}([X])=[\tau_n^-X]+(-1)^n[\Hom_{\Lambda^{\mathrm{op}}}(DX, \Lambda)]$.
\item[(g)] 
If $X\in\mcc^0$ is without injective direct summands, then $c^{-1}([X])=[\tau_n^-X]$. 
\item[(h)] 
If $X\in\mcc^0$ is injective, then $c^{-1}([X])=(-1)^n[\Hom_{\Lambda^{\mathrm{op}}}(DX, \Lambda)]$.
\end{itemize}
\end{proposition}
\begin{proof}
We only prove $(\rm{a})$, $(\rm{b})$, $(\rm{c})$ and $(\rm{d})$. The other parts follow by duality.
 
$(\rm{a})$ 
Consider the following projective resolution of $X$.
\begin{equation}
0\to P_n \overset{f_{n}}{\longrightarrow} P_{n-1}\to \cdots\to P_1  \overset{f_1}{\longrightarrow} P_0 \overset{f_0}{\longrightarrow} X\to0.
\end{equation}
We have $[X]= \sum_{i=0}^n(-1)^{i}[P_i]$ in $\mathrm{K}_0(\modd\Lambda)$ (for example see \cite[Proposition $\mathrm{VIII}$.4.1]{Bas}). Put $K_{-1}\coloneqq X$ and $K_i\coloneqq \ker f_i$ for any $i\in\{0,\dots,n-1\}$. Clearly, $K_{n-1}= P_n$. Consider the short exact sequence $0\to K_i\to P_i\to K_{i-1}\to 0$ for any $i\in\{0,\dots,n-1\}$. By applying $(-)^*=\Hom_\Lambda(-,\Lambda)$, we obtain
\begin{align*}
&0\to K_{i-1}^*\to P_i^*\to K_{i}^*\to\Ext^1_\Lambda(K_{i-1},\Lambda)\to \Ext^1_\Lambda(P_i,\Lambda)\to \Ext^1_\Lambda(K_i,\Lambda)\to\Ext^2_\Lambda(K_{i-1},\Lambda)\\
&\to \Ext^2_\Lambda(P_i,\Lambda)\to \Ext^2_\Lambda(K_i,\Lambda)\to\Ext^3_\Lambda(K_{i-1},\Lambda)\to\cdots
\end{align*}
It is obvious that $\Ext^j_\Lambda(P_i,\Lambda)=0$, for any $i\in\{0,\dots,n\}$ and $j>0$. Therefore, we have
\begin{align*}
\Ext^j_\Lambda(K_i,\Lambda)\cong\Ext^{j+1}_\Lambda(K_{i-1},\Lambda), \quad \forall\, i\in\{0,\dots,n-1\}\,, j\in\{1,\dots,n-1\},
\end{align*}
and easily we  obtain
\begin{equation}\label{q00}
\Ext^1_\Lambda(K_i,\Lambda)\cong\Ext^{i+2}_\Lambda(X,\Lambda), \quad \forall\, i\in\{0,\dots,n-2\}.
\end{equation}
 Consider the exact sequence $0\to K_{i-1}^*\to P_i^*\to K_{i}^*\to\Ext^1_\Lambda(K_{i-1},\Lambda)\to0$ for any $i\in\{0,\dots,n-1\}$. By applying $D$, we obtain
 $$0\to D\Ext^1_\Lambda(K_{i-1},\Lambda)\to DK_{i}^*\to DP_i^*\to DK_{i-1}^*\to0.$$
 So we have
 \begin{equation}\label{q01}
[DP_i^*]=[DK_{i-1}^*]+[DK_{i}^*]- [D\Ext^1_\Lambda(K_{i-1},\Lambda)] \,\,\text{in}\,\, \mathrm{K}_0(\modd\Lambda),\quad \forall\, i\in\{0,\dots,n-1\}. 
\end{equation}
Therefore,
\begin{alignat*}{2}
c([X]) &=c\big( \sum_{i=0}^n(-1)^{i}[P_i]\big) &\\
&=(-1)^n\big( \sum_{i=0}^n(-1)^{i}[I_i]\big)&\text{\quad (By Lemma \ref{pi})}\\
& =(-1)^n\big( \sum_{i=0}^n(-1)^{i}[DP_i^*]\big)&\\
& =(-1)^n\Big( \sum_{i=0}^{n-1}(-1)^{i}\big([DK_{i-1}^*]+[DK_{i}^*]- [D\Ext^1_\Lambda(K_{i-1},\Lambda)]\big)+(-1)^n[DP_n^*]\Big) & \text{\quad(By \eqref{q01})}\\
&=(-1)^n\Big((-1)^{0}\big([DK_{-1}^*]+[DK_{0}^*]- [D\Ext^1_\Lambda(K_{-1},\Lambda)]\big)+(-1)^{1}\big([DK_{0}^*]&\\
&+[DK_{1}^*]- [D\Ext^1_\Lambda(K_{0},\Lambda)]\big)+\cdots+(-1)^{n-1}\big([DK_{n-2}^*]+[DK_{n-1}^*]&\\
&- [D\Ext^1_\Lambda(K_{n-2},\Lambda)]\big)+(-1)^n[DP_n^*]\Big) &\\
& =(-1)^n\big([DX^*]-[D\Ext^1_\Lambda(X,\Lambda)]+ \sum_{i=0}^{n-2}(-1)^{i+2} [D\Ext^1_\Lambda(K_{i},\Lambda)]\big)&\\
& = (-1)^n\big([DX^*]-[D\Ext^1_\Lambda(X,\Lambda)]+ \sum_{i=0}^{n-2}(-1)^{i+2} [D\Ext^{i+2}_\Lambda(X,\Lambda)]\big) &\text{\quad (By \eqref{q00})}\\
&= (-1)^n\big([DX^*]-[D\Ext^1_\Lambda(X,\Lambda)]+ \sum_{i=2}^{n}(-1)^{i} [D\Ext^{i}_\Lambda(X,\Lambda)]\big)&\\
&=(-1)^n\big(\sum_{i=0}^{n}(-1)^{i} [D\Ext^{i}_\Lambda(X,\Lambda)]\big) &\\
&=\sum_{i=0}^{n}(-1)^{i+n} [D\Ext^{i}_\Lambda(X,\Lambda)].&
\end{alignat*}

$(\rm{b})$ By Proposition \ref{pic}, we have $\Ext^{i}_\Lambda(X,\Lambda)=0$ for any $i\in\{1,\dots,n-1\}$ and so the result follows.

$(\rm{c})$
Assume that $\Lambda$ is $n$-representation infinite. Then by Proposition \ref{Thio}$(\rm{f})$, $\mcc^0=\mathscr{P}\vee\mathscr{R}\vee\mathscr{I}$. If $X\in\mathscr{R}\vee\mathscr{I}$, then the result follows from parts $\rm{d}$ and $\rm{e}$ of Proposition \ref{Thio} and the previous part. Assume that $X\in\mathscr{P}$ is without projective direct summands. Consider the decomposition $X=\bigoplus_{i=1}^sX_i$ of $X$ into indecomposable summands. $X_i$ is non-projective for any $i\in\{1,\dots,s\}$. So by Theorem \ref{n-ext}$(\mathrm{a})$, there are an integer $t_i> 0$ and an indecomposable projective $\Lambda$-module $P_{i_j}$ such that $X_i\cong\nu_n^{-{t_i}}P_{i_j}\cong\tau_n^{-{t_i}}P_{i_j}$.   By Proposition \ref{thio0}, we have $$\Hom_\Lambda(X,\Lambda)=\bigoplus_{i=1}^s\Hom_\Lambda(\tau_n^{-{t_i}}P_{i_j},\Lambda)=\bigoplus_{i=1}^s\Hom_{\md_\Lambda}(\nu_n^{-{t_i}}P_{i_j},\Lambda)=0,$$ and the result follows from the previous part. Assume that $\Lambda$ is $n$-representation finite. Then by Theorem \ref{th1}, $\mcc ^0=\mathscr{P}$ and the result follows similarly.

$(\rm{d})$ $X$ is projective and so $\tau_nX=0$. Then the result follows from the part $(\rm{b})$.
\end{proof}

On the other hand, we know that $\Phi$ gives the action of $\nu_n$ on the Grothendieck group $\mathrm{K}_0(\md_\Lambda)$. Therefore, we have the following proposition.
\begin{proposition}\label{pro05}
Let $\Lambda$ be an $n$-hereditary algebra and $X\in\ind\mcc^0$. Then the following hold.
\begin{itemize}
\item[(a)]
If $X$ is non-projective, then $\Phi(\underline{\dim}X)=\underline{\dim}(\tau_nX)$. 
\item[(b)] 
If $X$ is non-injective, then $\Phi^{-1}(\underline{\dim}X)=\underline{\dim}(\tau_n^-X)$. 
\end{itemize}
\end{proposition}
\begin{proof}
We only prove $(\rm{a})$, the proof of $(\rm{b})$ is similar. Proposition \ref{tunu} implies that $\nu_n X\cong\tau_n X$. Therefore, we have
$$\Phi(\underline{\dim}X)=\underline{\dim}(\nu_nX)=\underline{\dim}(\tau_nX).$$
\end{proof}
As a consequence of Propositions \ref{pro04} and \ref{pro05}, we have the following fact that is well-known for $n=1$ (see, \cite[Proposition $\rm{XI}.1.1.2$]{SS}).
\begin{proposition}\label{fim}
Let $\Lambda$ be an $n$-representation infinite algebra. 
\begin{itemize}
\item[(a)]
If $X\in\ind(\mathscr{I}\vee\mathscr{R})$, then for any $t\geq 0$ we have 
$\Phi^t(\underline{\dim}X)=\underline{\dim}(\tau_n^tX)$.
\item[(b)] 
If $X\in\ind(\mathscr{P}\vee\mathscr{R})$, then for any $t\geq 0$ we have 
$\Phi^{-t}(\underline{\dim}X)=\underline{\dim}(\tau_n^{-t}X)$.
\end{itemize}
\end{proposition}
\begin{proof}
We only prove $(\rm{a})$, the proof of $(\rm{b})$ is similar. By parts $(\rm{d})$ and $(\rm{e})$ of Proposition \ref{Thio}, $(\mathscr{I}\vee\mathscr{R})\cap\mathscr{P}=0$. Therefore, $\tau_n^tX$ is non-projective, for any $t\geq 0$ and the result follows from Proposition \ref{pro05}$(\rm{a})$.
\end{proof}
\begin{example}\label{ex01}
	Consider the following non-Dynkin quivers.		
	$$
	\begin{array}{ccc}
		\xymatrix{
			Q_1:\,\bullet\ar@2{->}[rr]& &\bullet \ar[rr]&&\bullet}
		&
		\text{and} 
		& 
		Q_2:\,\xymatrix{
			\bullet\ar@2{->}[rr]& &\bullet}
	\end{array}$$		
	By \cite[Example 2.12]{HIO} the tensor product $KQ_1\otimes_K KQ_2$ is a $2$-representation infinite algebra. By \cite[Lemma 1.3]{Le}, $KQ_1\otimes_K KQ_2\cong KQ/ \mathcal{I}$
	such that $Q$ is the quiver
	$$\xymatrix{
		6 \ar@2{->}[rr]^{x_1}_{y_1}\ar@2{->}[dd]^{v_1}_{w_1}& &5 \ar[rr]^{z_1}\ar@2{->}[dd]^{v_2}_{w_2}&&4\ar@2{->}[dd]^{v_3}_{w_3}\\
		&&&&\\
		3 \ar@2{->}[rr]^{x_2}_{y_2} && 2 \ar[rr]^{z_2}&&1}$$ 
	and $\mathcal{I}$ is the ideal generated by the relations
	$x_1w_2=w_1x_2,\, x_1v_2=w_1y_2,\, y_1w_2=v_1x_2,\, y_1v_2=v_1y_2,\, z_1v_3=w_2z_2$ and $z_1v_3=v_2z_2$. 
	 The Cartan matrix and the Coxeter matrix of $\Lambda$ are the matrices
	 $$\begin{array}{cc}
	     C_\Lambda =\begin{bmatrix}                                                                                     
	 		1&1&2&2&2&4\\
	 		0&1&2&0&2&4\\
	 		0&0&1&0&0&2\\
	 		0&0&0&1&1&2\\
	 		0&0&0&0&1&2\\
	 		0&0&0&0&0&1
	 	\end{bmatrix}, & 	\Phi =(-1)^2C_\Lambda^t C_\Lambda^{-1}=\begin{bmatrix}                                                                                     
	 					1&-1&0&-2&2&0\\
	 					1&0&-2&-2&0&4\\
	 					2&0&-3&-4&0&6\\
	 					2&-2&0&-3&3&0\\
	 					2&0&-4&-3&0&6\\
	 					4&0&-6&-6&0&9
	 				\end{bmatrix}.
	   \end{array}$$
We know that $\mathscr{I}=\add\{\tau_2^{i}(D\Lambda)\,|\,i\geq 0\}$. We can apply Proposition \ref{pro05} and compute $\underline{\dim}(\tau_2^iX)$, for any $X\in\ind\mathscr{I}$ and $i\geq 0$. For instance,
 \begin{align*}
 \underline{\dim}I_1&=\begin{bmatrix}                  
 	1\\
 	1\\
 	2\\
 	2\\
 2\\
 4
 \end{bmatrix}
 \overset{\Phi}{\Rightarrow}	
  \underline{\dim}(\tau_2I_1)=\begin{bmatrix}                                 	0\\
  	9\\
  	12\\
  	0\\
  	12\\
 16
 \end{bmatrix}\overset{\Phi}{\Rightarrow}
 \underline{\dim}(\tau_2^2I_1)=\begin{bmatrix}                                                                                     
 	15\\
 	40\\
 	60\\
 	18\\
 	48\\
 	72
 \end{bmatrix}\overset{\Phi}{\Rightarrow}
 \underline{\dim}(\tau_2^3I_1)=\begin{bmatrix}                                                                                     
 	35\\
 	147\\
 	210\\
 	40\\
 	168\\
 	240
 \end{bmatrix}\overset{\Phi}{\Rightarrow}\cdots	\\
 \underline{\dim}I_2&=\begin{bmatrix}                                                                                     
 	0\\
 	1\\
 	2\\
 	0\\
 	2\\
 	4
 \end{bmatrix}
 \overset{\Phi}{\Rightarrow}	
 \underline{\dim}(\tau_2I_2)=\begin{bmatrix}                                 	3\\
 	12\\
 	18\\
 	4\\
 	16\\
 	24
 \end{bmatrix}\overset{\Phi}{\Rightarrow}
 \underline{\dim}(\tau_2^2I_2)=\begin{bmatrix}                                                                                     
 	15\\
 	55\\
 	80\\
 	18\\
 	66\\
 	96
 \end{bmatrix}\overset{\Phi}{\Rightarrow}\cdots	\\
 \underline{\dim}I_3&=\begin{bmatrix}                                                                                     
 	0\\
 	0\\
 	1\\
 	0\\
 	0\\
 	2
 \end{bmatrix}
 \overset{\Phi}{\Rightarrow}	
 \underline{\dim}(\tau_2I_3)=\begin{bmatrix}                                 	0\\
 	6\\
 	9\\
 	0\\
 	8\\
 	12
 \end{bmatrix}\overset{\Phi}{\Rightarrow}
 \underline{\dim}(\tau_2^2I_3)=\begin{bmatrix}                                                                                     
 	10\\
 	30\\
 	45\\
 	12\\
 	36\\
 	54
 \end{bmatrix}\overset{\Phi}{\Rightarrow}\cdots\\
 \underline{\dim}I_4&=\begin{bmatrix}                                                                                     
 	0\\
 	0\\
 	0\\
 	1\\
 	1\\
 	2
 \end{bmatrix}
 \overset{\Phi}{\Rightarrow}	
 \underline{\dim}(\tau_2I_4)=\begin{bmatrix}                                 	0\\
 	6\\
 	8\\
 	0\\
 	9\\
 	12
 \end{bmatrix}\overset{\Phi}{\Rightarrow}
 \underline{\dim}(\tau_2^2I_4)=\begin{bmatrix}                                                                                     
 	12\\
 	32\\
 	48\\
 	15\\
 	40\\
 	60
 \end{bmatrix}\overset{\Phi}{\Rightarrow}\cdots\\
  \underline{\dim}I_5&=\begin{bmatrix}                                                                                     
 	0\\
 	0\\
 	0\\
 	0\\
 	1\\
 	2
 \end{bmatrix}
 \overset{\Phi}{\Rightarrow}	
 \underline{\dim}(\tau_2I_5)=\begin{bmatrix}                                 	2\\
 	8\\
 	12\\
 	3\\
 	12\\
 	18
 \end{bmatrix}\overset{\Phi}{\Rightarrow}
 \underline{\dim}(\tau_2^2I_5)=\begin{bmatrix}                                                                                     
 	12\\
 	44\\
 	64\\
 	15\\
 	55\\
 	80
 \end{bmatrix}\overset{\Phi}{\Rightarrow}\cdots
  \end{align*}
\begin{align*}
 \underline{\dim}I_6&=\begin{bmatrix}                                                                                     
 	0\\
 	0\\
 	0\\
 0\\
 	0\\
 	1
 \end{bmatrix}
 \overset{\Phi}{\Rightarrow}	
 \underline{\dim}(\tau_2I_6)=\begin{bmatrix}                                 	0\\
 	4\\
 	6\\
 	0\\
 	6\\
 	9
 \end{bmatrix}\overset{\Phi}{\Rightarrow}
 \underline{\dim}(\tau_2^2I_6)=\begin{bmatrix}                                                                                     
 	8\\
 	24\\
 	36\\
 	10\\
 	30\\
 	45
 \end{bmatrix}\overset{\Phi}{\Rightarrow}\cdots
 \end{align*}
	\end{example}
	
\section{Modules determined by their composition factors}\label{s5}
In this section, first we recall the concept of quadratic forms. Then we show that indecomposable modules in $\mathscr{P}\vee\mathscr{I}$ are uniquely determined up to isomorphism by their dimension vectors for an $n$-hereditary algebra $\Lambda$ with odd $n$.
\begin{definition}[\protect{\cite[Definition $\mathrm{III}$.3.11]{ASS}}]
Let $\Lambda$ be a basic finite dimensional $K$-algebra with $\gd\Lambda<\infty$ and Cartan matrix $C_\Lambda$. The \textit{Euler characteristic} of $\Lambda$ is the following $\mathbb{Z}$-bilinear form.
\begin{align*}
\langle-,-\rangle:\mathbb{Z}^m\times\mathbb{Z}^m&\to\mathbb{Z}\\
 (x,y)&\mapsto x^t(C_\Lambda^{-1})^t y
\end{align*}
Also, the corresponding quadratic form to $\langle-,-\rangle$ is called \textit{the Euler quadratic form} of $\Lambda$ and is defined as follows.
\begin{align*}
q:\mathbb{Z}^m&\to\mathbb{Z}\\
 x&\mapsto \langle x,x\rangle
\end{align*}
\end{definition}

On the other hand, another quadratic form corresponds to any finite dimensional $K$-algebra of finite global dimension. Let $\gd\Lambda\leq n$. For any $X,Y\in\modd\Lambda$, set $$B(X,Y)\coloneqq\sum_{i=0}^n(-1)^i\dim_K\Ext^i_\Lambda(X,Y).$$ It is easy to see that
\begin{align}\label{tag1}
B:\mathrm{K}_0(\modd\Lambda)\times\mathrm{K}_0(\modd\Lambda)&\to\mathbb{Z},\\
\notag
 ([X],[Y])&\mapsto B(X,Y)
\end{align}
is a well-defined bilinear form. The quadratic form associated
with $B$ is
\begin{align*}
\chi:\mathrm{K}_0(\modd\Lambda)&\to\mathbb{Z},\\
 x&\mapsto B(x,x)
\end{align*}
which is called the \textit{homological
quadratic form} of $\Lambda$ (see \cite[Section $\mathrm{VIII}$.3]{ARS}). This form was introduced by Ringel \cite{Ri}.

Now, we recall a useful relationship between the homological
quadratic form and the Euler quadratic form in the following theorem.
\begin{proposition}[\protect{\cite[Proposition $\mathrm{III}$.3.13)]{ASS}}]\label{th2}
Let $\gd\Lambda<\infty$. Then, the following identities hold for any $X,Y\in\modd\Lambda$.
\begin{enumerate}
\item[(a)]
$\langle\underline{\dim}X,\underline{\dim}Y\rangle=B([X],[Y])$.
\item[(b)]
$q(\underline{\dim}X)=\chi([X])$.
\end{enumerate}
\end{proposition}
Let $\Lambda$ be an $n$-hereditary algebra. Clearly, the Euler quadratic form and the homological quadratic form are well-defined for $\Lambda$.
\begin{example}
Let $\Lambda$ be the $2$-representation infinite algebra in Example \ref{ex01}.  
Then the Euler quadratic form $	q:\mathbb{Z}^6\to\mathbb{Z}$ of $\Lambda$ is given by
$$q(x)=x_1^2+x_2^2+x_3^2+x_4^2+x_5^2+x_6^2-x_1x_2-2x_2x_3-2x_1x_4+2x_1x_5-2x_2x_5-x_4x_5+4x_2x_6-2x_3x_6-2x_5x_6,$$
for any $x=(x_1,\dots,x_6)\in \mathbb{Z}^6$.
\end{example}
As in the classical case, we have the following relation between Coxeter matrices and the Euler forms of $n$-hereditary algebras (for the proof of the classical case, see \cite[Lemma $\mathrm{III}$.3.16]{ASS}).
\begin{lemma}\label{x-fi}
Let $\Lambda$ be an $n$-hereditary algebra with the Coxeter matrix $\Phi$. Then the following hold for any $x,y\in\mathbb{Z}^m$.
\begin{enumerate}
\item[(a)]
	$\langle x,y \rangle=(-1)^n\langle\ y,\Phi(x) \rangle=\langle\ \Phi(x),\Phi(y) \rangle$.
	\item[(b)]
	$\langle x,y \rangle=(-1)^n\langle\ \Phi^{-1}(y),x \rangle=\langle\ \Phi^{-1}(x),\Phi^{-1}(y) \rangle$.
\end{enumerate} 
\end{lemma}
\begin{remark}\label{fo}
Let $\Lambda$ be an $n$-hereditary algebra. Consider the bilinear form $B$. Proposition \ref{pic} implies that
\begin{align}\label{x}
B([X],[Y])&=\sum_{i=0}^n(-1)^i\dim_K\Ext_\Lambda^i(X,Y)\nonumber\\
&=\dim_K\Hom_\Lambda(X,Y)+(-1)^n\dim_K\Ext_\Lambda^n(X,Y),
\end{align}
for any $X\in\mathscr{P}\vee\mathscr{I}$ and $Y\in\mcc^0$ or $X\in\mcc^0$ and $Y\in\mathscr{P}\vee\mathscr{I}$. Specially, if $\Lambda$ is $n$-representation finite, then Theorem \ref{th1} implies that $\mcc ^0=\mathscr{P}=\mathscr{I}$. Therefore, the equality \eqref{x} holds for any $X,Y\in\mcc^0$.
\end{remark}
In the representation theory of hereditary algebras, classifying indecomposable
modules up to isomorphism is one of the most important parts of the theory. As one of the main results, we know that indecomposable preprojective and preinjective
modules over hereditary algebras are uniquely determined by their dimension vectors. Now, we intend to extend this result to $n$-hereditary algebras with odd $n$. In the following theorem, we show that for such algebras, indecomposable modules in $\mathscr{P}\vee\mathscr{I}$ are uniquely determined up to isomorphism by their dimension vectors.

\begin{theorem}\label{dv}
Let $\Lambda$ be an $n$-hereditary algebra with odd $n$. Then the following hold.
\begin{enumerate}
\item[(a)]
For $X,Y\in\ind\mathscr{P}$, $X\cong Y$ if and only if $\underline{\dim}X=\underline{\dim}Y$.
\item[(b)]
For $X,Y\in\ind\mathscr{I}$, $X\cong Y$ if and only if $\underline{\dim}X=\underline{\dim}Y$.
\end{enumerate}
\end{theorem}
\begin{proof}
We only prove $(\mathrm{a})$, the proof of $(\mathrm{b})$ is similar.

\noindent
It is clear that if $X\cong Y$, then $\underline{\dim}X=\underline{\dim}Y$. Assume that $\underline{\dim}X=\underline{\dim}Y$.
Since $n$ is odd, by Remark \ref{fo}, we have
\begin{align*}
B([X],[Y])&=\dim_K\Hom_\Lambda(X,Y)-\dim_K\Ext_\Lambda^n(X,Y),\\
B([Y],[X])&=\dim_K\Hom_\Lambda(Y,X)-\dim_K\Ext_\Lambda^n(Y,X),\\
\chi([X])=B([X],[X])&=\dim_K\Hom_\Lambda(X,X)-\dim_K\Ext_\Lambda^n(X,X).
\end{align*}
On the other hand, by Proposition \ref{n-ext}$(\mathrm{a})$, we have $\Ext_\Lambda^n(X,X)=0$ and so
$$\chi([X])=\dim_K\Hom_\Lambda(X,X).$$
By the assumption, $\underline{\dim}X=\underline{\dim}Y$ and so $[X]=[Y]$ (see \cite[Theorem $\mathrm{III}$.3.5]{ASS}). Therefore, we have
\begin{align*}
B([X],[Y])=B([X],[X])=\chi([X]),\\
B([Y],[X])=B([X],[X])=\chi([X]).
\end{align*}
Namely,
\begin{align*}
\dim_K\Hom_\Lambda(X,Y)-\dim_K\Ext_\Lambda^n(X,Y)=\dim_K\Hom_\Lambda(X,X),\\
\dim_K\Hom_\Lambda(Y,X)-\dim_K\Ext_\Lambda^n(Y,X)=\dim_K\Hom_\Lambda(X,X).
\end{align*}
Since $X\neq0$, $\dim_K\Hom_\Lambda(X,X)>0$ and so $\Hom_\Lambda(X,Y)\neq0$ and $\Hom_\Lambda(Y,X)\neq0$. By the assumption, $X,Y\in\ind\mathscr{P}$ and so by Proposition \ref{n-ext}$(\mathrm{a})$, there are unique integers $i,j\geq 0$ and unique indecomposable projective $\Lambda$-modules $P,Q$ such that $X\cong\nu_n^{-i}P\cong\tau_n^{-i}P$ and $Y\cong\nu_n^{-j}Q\cong\tau_n^{-j}Q$. Now, we have
\begin{align*}
0\neq\Hom_\Lambda(X,Y)\cong\Hom_\Lambda(\tau_n^{-i}P,\tau_n^{-j}Q)\cong\Hom_{\md_\Lambda}(\nu_n^{-i}P,\nu_n^{-j}Q),\\
0\neq\Hom_\Lambda(Y,X)\cong\Hom_\Lambda(\tau_n^{-j}Q,\tau_n^{-i}P)\cong\Hom_{\md_\Lambda}(\nu_n^{-j}Q,\nu_n^{-i}P).
\end{align*}
Proposition \ref{thio0} implies that $-i\geq-j$ and $-j\geq-i$. Hence $i=j$ and so $X\cong\nu_n^{-i}P\cong\tau_n^{-i}P$ and $Y\cong\nu_n^{-i}Q\cong\tau_n^{-i}Q$. It is obvious that $\tau_n^{-k} P$ and $\tau_n^{-k} Q$ are non-injective, for any $k$ with $0\leq k<i$. Therefore by Proposition \ref{pro05}$(\mathrm{b})$, we have 
\begin{align*}
\underline{\dim}X=\underline{\dim}\tau_n^{-i}(P)=\Phi^{-i}(\underline{\dim}P),\\
\underline{\dim}Y=\underline{\dim}\tau_n^{-i}(Q)=\Phi^{-i}(\underline{\dim}Q).
\end{align*}
By the assumption, $\underline{\dim}X=\underline{\dim}Y$ so $\Phi^{-i}(\underline{\dim}P)=\Phi^{-i}(\underline{\dim}Q)$. By applying $\Phi^i$, we have $\underline{\dim}P=\underline{\dim}Q$.
 By Proposition \ref{car}, $P\cong Q$ and so $X\cong Y$.
\end{proof}
If we put $n=1$ in the above theorem, then we obtain the classical result for hereditary algebras. In fact, if $\Lambda$ is a hereditary algebra 
	and $X, Y$ are indecomposable $\Lambda$-modules in the preprojective (or preinjective) component, then $X\cong Y$ if and only if $\underline{\dim}X=\underline{\dim}Y$. For the proof of this result we refer to \cite[Corollary $\mathrm{VIII}.2.3$]{ARS} or \cite[Proposition 2.4]{AP}. Note that the proofs of the classical case are heavily based on the properties of hereditary algebras. For example, the fact that submodules of projective modules over a hereditary algebra remain projective is very important in the proofs of the classical case. Our method for the proof of Theorem \ref{dv} is different from the classical approach.

We know that if $\Lambda$ is $n$-representation finite, then $\mcc ^0=\mathscr{P}=\mathscr{I}$. Therefore, as a corollary of Theorem \ref{dv}, all modules in the $n$-cluster tilting subcategory $\mcc^0$ are uniquely determined up to isomorphism by their dimension vectors. This result was also proved using a different method in \cite{Re}. In the other attempt, Mizuno proved this fact for any arbitrary $n$ under a specific condition on $\Lambda$. He assumed that the Gabriel quiver of $\Lambda$ is acyclic (see \cite[Theorem 4.12]{M}). 
\section{Coxeter matrix and $n$-representation finiteness}\label{s6}
In this section, as a consequence of Theorem \ref{dv}, we give a criterion to describe whether an $n$-hereditary algebra is $n$-representation finite. We show that for an $n$-hereditary algebra $\Lambda$ with odd $n$, the condition $\Phi^d = 1$ holds for some positive integer $d$ if and only if $\Lambda$ is $n$-representation finite. First, we prove that one direction of this implication is true for any arbitrary $n$. In fact, if $\Lambda$ is $n$-representation finite then  $\Phi^d=1$ for some positive integer $d$. Then, we prove the converse for the case that $n$ is odd.
\begin{proposition}\label{fi-re} 
	Let $\Lambda$ be an $n$-hereditary algebra. If $\Lambda$ is $n$-representation finite, then there exists a positive integer $d$ such that $\Phi^d=1$.                                     
\end{proposition}  
\begin{proof} 
	Since $\gd\Lambda\leq n$, $\{\underline{\dim}I_1,\dots,\underline{\dim}I_m\}$ is a basis of $\mathbb{Z}^m$. So it is enough to show that there exists a positive integer $d$ such that $\Phi^d(\underline{\dim}I_i)=\underline{\dim}I_i$, for any $i\in\{1,\dots,m\}$. By Proposition \ref{per}$(\rm{a})$, there exist a permutation $\sigma$ on $\{1,\dots,m\} $ and non-negative integers $t_i$ such that $\tau_n^{t_i} I_i\cong P_{\sigma(i)}$, for any $i \in \{1,\dots,m\}$. Therefore, $\underline{\dim}(\tau_n^{t_i} I_i)= \underline{\dim}P_{\sigma(i)}$, for any $i$. For any $i\in \{1,\dots,m\}$ with $t_i=0$, by Lemma \ref{pi}, we have 
	\begin{equation}\label{q0}
		\Phi(\underline{\dim}I_i)=\Phi(\underline{\dim} P_{\sigma(i)})=(-1)^n\underline{\dim} I_{\sigma(i)}.
	\end{equation}
	Assume that $S_m$ is the permutations group on $\{1,\dots,m\}$. The order of any element of $S_m$ is finite, so suppose that the positive integer $r$ is the order of $\sigma$. Namely, $\sigma^r=1$.
	We have
	\begin{equation}\label{q1}
		\Phi^{2r}(\underline{\dim}I_i)=(-1)^{2rn}\underline{\dim} (I_{\sigma^{2r}(i)})=\underline{\dim} I_i.
	\end{equation}
	Now, assume that $i\in \{1,\dots,m\}$ with $t_i\neq0$. It is obvious that $\tau_n^{k} I_i$ is non-projective, for any $k$ with $0\leq k<t_i$. Therefore, by Proposition \ref{pro05}$(\rm{a})$, we have
	$$\Phi^{t_i}(\underline{\dim}I_i)=\underline{\dim}(\tau_n^{t_i} I_i)=\underline{\dim} P_{\sigma(i)}.$$
	By applying $\Phi$ and using Lemma \ref{pi}, we have
	\begin{equation}\label{q03}
		\Phi^{t_i+1}(\underline{\dim}I_i)=\Phi(\underline{\dim} P_{\sigma(i)})=(-1)^n\underline{\dim} I_{\sigma(i)}.
	\end{equation}
	If $t_{\sigma(i)}=0$, by \eqref{q0}, we have
	\begin{equation}\label{q04}
		\Phi^{t_{\sigma(i)}+1}(\underline{\dim}I_{\sigma(i)})=\Phi(\underline{\dim}I_{\sigma(i)})=(-1)^n\underline{\dim}I_{\sigma^2(i)}.
	\end{equation}
	If $t_{\sigma(i)}\neq0$, by \eqref{q03}, we have
	\begin{equation}\label{q05}
		\Phi^{t_{\sigma(i)}+1}(\underline{\dim}I_{\sigma(i)})=(-1)^n\underline{\dim}I_{\sigma^2(i)}.
	\end{equation}
	So we apply $\Phi^{t_{\sigma(i)}+1}$ on \eqref{q03} and by \eqref{q04} and \eqref{q05} we obtain
	\begin{align*}
		\Phi^{t_{\sigma(i)}+1}\big(\Phi^{t_i+1}(\underline{\dim}I_i)\big)&=(-1)^n\Phi^{t_{\sigma(i)}+1}(\underline{\dim} I_{\sigma(i)})\\
		\Phi^{t_{\sigma(i)}+t_i+2}(\underline{\dim}I_i)&=(-1)^n\big((-1)^n\underline{\dim} I_{\sigma^2(i)}\big)=\underline{\dim} I_{\sigma^2(i)}.
	\end{align*}
	Put $h_i\coloneqq t_{\sigma(i)}+t_i+2$ and $\alpha\coloneqq\sigma^2$. So, we have
	$$\Phi^{h_i}(\underline{\dim}I_i)=\underline{\dim} I_{\alpha(i)}.$$
	Suppose that the positive integer $s$ is the order of $\alpha$.
	Put $u_i\coloneqq h_{\alpha^{s-1}(i)}+\cdots+ h_{\alpha(i)}+h_i$. Hence, for any $i\in \{1,\dots,m\}$ with $t_i\neq0$ we have
	\begin{equation}\label{q2}
		\Phi^{u_i}(\underline{\dim}I_i)=\underline{\dim} I_{\alpha^{s}(i)}=\underline{\dim} I_i.
	\end{equation}
		Put $d\coloneqq 2r+\displaystyle\sum_{i=1 \atop t_i\neq 0}^m u_i$. By \eqref{q1} and \eqref{q2}, it is obvious that $\Phi_\Lambda^d(\underline{\dim}I_i)=\underline{\dim}I_i$, for any $i\in \{1,\dots,m\}$ and 
	therefore, $\Phi^d=1$.
\end{proof} 
Note that, according to the above proof, the positive integer $d$ in Proposition \ref{fi-re} is even. Proposition \ref{fi-re} was also proved in \cite[Proposition 3.5]{M} with a quite different technique.                         

Let $\rm{Gl}(m,\mathbb{Z})$ be the group of $m\times m$ invertible matrices with integer entries. The problem of determining the order of an element $\Phi$ in $\rm{Gl}(m,\mathbb{Z})$ has a long history. When $\Phi$ has finite order, its order can be computed by using methods from linear algebra (for example see \cite{KP}).

By Proposition \ref{fi-re}, we give an upper bound on the number of indecomposable modules in the $n$-cluster tilting subcategory $\mcc^0$.
\begin{corollary}\label{co1} 
Let $\Lambda$ be an $n$-representation finite algebra. If $l$ is the number of indecomposable injective-projective $\Lambda$-modules then the number of indecomposable $\Lambda$-modules in $\mcc^0$ is at most $(m-l)d+l$, where $m$ is the number of isomorphism classes of simple $\Lambda$-modules and $d$ is the smallest positive integer that $\Phi^d=1$.
\end{corollary}
\begin{proof}	
	By Proposition \ref{per},  $A=\{\tau_n^{j}I_i\,|\,1\leq i\leq m , 0\leq j\leq t_i\}$ is the complete set of the isomorphism classes of indecomposable $\Lambda$-modules of $\mcc^0$ (we recall that $\tau_n^{t_i} I_i\cong P_{\sigma(i)}$, for any $i \in \{1,\dots,m\}$). Assume that $I_i$ is an indecomposable injective non-projective module. Then according to the proof of Proposition \ref{fi-re}, $t_i< d$ and so $A$ has at most $(m-l)d+l$ elements.
\end{proof}
\begin{theorem}\label{fi-fi}
	Let $\Lambda$ be an $n$-hereditary algebra with odd $n$. Then there exists a positive integer $d$ such that $\Phi^d=1$ if and only if $\Lambda$ is $n$-representation finite.
\end{theorem}
\begin{proof}
	Assume that $\Phi^d=1$, for a positive integer $d$. Then $\Phi^d(\underline{\dim}I_i)=\underline{\dim}I_i$, for any $i\in\{1,\dots,m\}$. If $\Lambda$ is $n$-representation infinite, Proposition \ref{fim}$(\rm{a})$ implies that $\Phi^d(\underline{\dim}I_i)=\underline{\dim}(\tau_n^d I_i)$, for any $i$. Therefore, $\underline{\dim}I_i=\underline{\dim}(\tau_n^d I_i)$, for any $i$ and by Theorem \ref{dv}$(\rm{b})$, $I_i\cong \tau_n^d I_i$, for any $i\in\{1,\dots,m\}$. But this contradicts Proposition \ref{Thio}$(\rm{b})$.
	
	The other side follows from Proposition \ref{fi-re} and the proof is complete.
\end{proof}
If we put $n=1$ in Theorem \ref{fi-fi}, we obtain the well-known result for the hereditary case that was proved by Platzeck and Auslander in \cite[Theorem 2.6]{AP}.  
\section{Quadratic forms of $n$-hereditary algebras}\label{s7}
In this section, first we introduce a Grothendieck group corresponding to $\mcc^0$, for an $n$-hereditary algebra $\Lambda$. Then, as an application of Theorem \ref{dv}, we prove that the homological quadratic form of the $n$-representation finite algebra $\Lambda$ is positive definite, for any odd $n$. Furthermore, we show that the converse of this fact does not hold for any arbitrary $n$ by giving counterexamples. 
\subsection{Grothendieck group of $\boldsymbol{\mcc^0}$}\label{g}
In this subsection, first we recall the concept of the Grothendieck group of $n$-abelian categories. For $n$-representation finite $n$-hereditary algebras, $\mcc^0$ is an $n$-abelian category. So, the Grothendieck group of $\mcc^0$ is well-defined. We introduce a Grothendieck group corresponding to $\mcc^0$, for an $n$-hereditary algebra $\Lambda$. We show that it coincides with the defined Grothendieck group in the $n$-representation finite case. Finally, we prove that there is a group isomorphism between Grothendieck group of $\mcc^0$ and the ordinary Grothendieck group of $\Lambda$.

In section \ref{s4}, we recalled the notion of Grothendieck group $\mathrm{K}_0(\mathcal{C})$ of the abelian category $\mathcal{C}$. In fact, the  quotient group
\begin{equation}
	\mathrm{K}_0(\mathcal{C})\coloneqq\mathrm{K}_0(\mathcal{C},0)/\langle[X]_0-[Y]_0+[Z]_0\,|\,0\to X\to Y\to Z\to0\text{ is a short exact sequence in\,\,} \mathcal{C}\rangle\notag
\end{equation}
is called \textit{Grothendieck group} of $\mathcal{C}$, where $\mathrm{K}_0(\mathcal{C},0)$ is the split Grothendieck group of $\mathcal{C}$ and $[X]_0$ is the corresponding element of $X\in\mathcal{C}$ in $\mathrm{K}_0(\mathcal{C},0)$. The above definition is generalized to the $n$-abelian categories. These categories that were introduced by Jasso are analogues of abelian categories from the point of view of higher homological algebra (see \cite{J} for the definition and properties of $n$-abelian categories).
Specially, according to \cite[Theorem 3.16]{J}, any $n$-cluster tilting subcategory $\mathcal{M}$ of $\modd\Lambda$ is $n$-abelian.
To define the Grothendieck group of an $n$-abelian category, we first recall the notion of $n$-exact sequences.

Let $\mathcal{C}$ be an additive category and $d_{n+1}:X_{n+1} \rightarrow X_n$ be a morphism in $\mathcal{C}$. An \textit{$n$-cokernel} of $d_{n+1}$ is a sequence
\begin{equation}
(d_n, \ldots, d_1): X_n \overset{d_n}{\longrightarrow} X_{n-1} \xrightarrow {d_{n-1}}\cdots \overset{d_{2}}{\longrightarrow} X_1 \overset{d_1}{\longrightarrow} X_{0} \notag
\end{equation}
in $\mathcal{C}$ such that for all $Y\in \mathcal{M}$
the induced sequence of abelian groups
\begin{align}
0 \rightarrow \Hom_{\mathcal{C}}(X_0,Y) \rightarrow \Hom_{\mathcal{C}}(X_1,Y) \rightarrow\cdots\rightarrow \Hom_{\mathcal{C}}(X_n,Y) \rightarrow \Hom_{\mathcal{C}}(X_{n+1},Y) \notag
\end{align}
is exact \cite[Definition 2.2]{J}. Dually, an \textit{$n$-kernel} of a morphism in $\mathcal{C}$ is defined. Also, a complex
\begin{equation}
X_{n+1} \xrightarrow {d_{n+1}} X_n \overset{d_n}{\longrightarrow} X_{n-1} \xrightarrow {d_{n-1}}\cdots \overset{d_{2}}{\longrightarrow} X_1 \overset{d_1}{\longrightarrow} X_{0} \notag
\end{equation}
in $\mathcal{C}$ is called an \textit{$n$-exact sequence} if $(d_n, \ldots, d_1)$ is an $n$-cokernel of $d_{n+1}$ and $(d_{n+1}, \ldots, d_2)$ is an $n$-kernel of $d_1$ \cite[Definition 2.4]{J}.
\begin{definition}$($\cite[Definition 1.3]{Re}$)$\label{g1}
Let $\mathcal{C}$ be an $n$-abelian category. Then the \textit{Grothendieck group} of $\mathcal{C}$ is defined as
\begin{eqnarray}
\mathrm{K}_0(\mathcal{C})\coloneqq\mathrm{K}_0(\mathcal{C},0)/\langle\sum^{n+1}_{i=0}(-1)^i[X_i]_0\,|\,0\to X_{n+1}\to\cdots \to X_{0}\to 0 \text{\ is an $n$-exact sequence in $\mathcal{C}$}\rangle. \nonumber
\end{eqnarray}
\end{definition}

Now, let $\Lambda$ be an $n$-hereditary algebra. It is obvious that $\mcc^0$ is an essentially small additive category, so we can consider the split Grothendieck group $\mathrm{K}_0(\mcc^0,0)$ of $\mcc^0$ and define the \textit{Grothendieck group of $\mcc^0$} as the quotient group 
\begin{align*}
\mathrm{K}_0(\mcc^0)\coloneqq\mathrm{K}_0(\mcc^0,0)/\langle & \sum^{n+1}_{i=0}(-1)^i[X_i]_0\,|\,0\to X_{n+1}\to\cdots \to X_{0}\to 0\\
&\text{\,\,is an exact sequence with $n+2$ terms in $\mcc^0$}\rangle.
\end{align*}
\begin{remark}\label{rem11}
Let $\Lambda$ be an $n$-hereditary algebra. The definition of $\mathrm{K}_0(\mcc^0)$ leads to the following results.
\begin{enumerate}
\item
For the case $n=1$, $\mcc^0=\modd\Lambda$ and hence $\mathrm{K}_0(\mcc^0)=\mathrm{K}_0(\modd\Lambda)$.
\item
If $\Lambda$ is an $n$-representation finite algebra, then Theorem \ref{th1} implies that $\mcc^0$ is an $n$-cluster tilting subcategory of $\modd\Lambda$. For the $n$-abelian category $\mcc^0$, by Definition \ref{g1}, there is the following Grothendieck group.
\begin{align*}
\mathrm{K}_0(\mcc^0,0)/\langle\sum^{n+1}_{i=0}(-1)^i[X_i]_0\,|\,0\to X_{n+1}\to\cdots \to X_{0}\to 0 \text{ is an $n$-exact sequence in $\mcc^0$}\rangle.
\end{align*} 
We denote this Grothendieck group by $\mathrm{K}_0^\prime(\mcc^0)$ and we put
\begin{align*}
I^\prime &\coloneqq\langle\sum^{n+1}_{i=0}(-1)^i[X_i]_0\,|\,0\to X_{n+1}\to\cdots \to X_{0}\to 0\text{ is an $n$-exact sequence in $\mcc^0$}\rangle,\nonumber\\
I&\coloneqq\langle\sum^{n+1}_{i=0}(-1)^i[X_i]_0\,|\,0\to X_{n+1}\to\cdots \to X_{0}\to 0\text{ is an exact sequence with $n+2$ terms in $\mcc^0$}\rangle.\nonumber
\end{align*}
By \cite[Lemma 3.5]{I4} and \cite[Lemma 5.1]{Jo}, exact sequences with $n+2$ terms in $\mcc^0$ are exactly $n$-exact sequences in $\mcc^0$. Therefore, we have $I^\prime=I$ and so $\mathrm{K}_0^\prime(\mcc^0)=\mathrm{K}_0(\mcc^0)$.
Since $\modd\Lambda$ has an $n$-cluster tilting object, $\mathrm{K}_0^\prime(\mcc^0)\cong\mathrm{K}_0(\modd\Lambda)$ (see \cite[Theorem 3.7]{Re} or \cite[Theorems 3.9 and 3.11]{DN1}). 
Therefore, $\mathrm{K}_0(\mcc^0)\cong\mathrm{K}_0(\modd\Lambda)$. 
\end{enumerate}
\end{remark}
As mentioned in Remark \ref{rem11}, we have an isomorphism $\mathrm{K}_0(\mcc^0) \cong \mathrm{K}_0(\modd \Lambda)$ for some $n$-hereditary algebras $\Lambda$. We claim that this isomorphism holds for any arbitrary $n$-hereditary algebra.

For any $X\in\modd\Lambda$ and $X\in\mcc^0$, we denote by $[X]$ and $[X]_{\mcc^0}$ the corresponding elements in $\mathrm{K}_0(\modd\Lambda)$ and  $\mathrm{K}_0(\mcc^0)$, respectively.
\begin{theorem}\label{T1}
Let $\Lambda$ be an $n$-hereditary algebra. Then, $\mathrm{K}_0(\mcc^0)\cong\mathrm{K}_0(\modd\Lambda)$.
\end{theorem}
\begin{proof}
There exists a natural morphism
\begin{align*}
i:\mathrm{K}_0(\mcc^0,0) & \rightarrow \mathrm{K}_0(\modd\Lambda,0).\\
 [X]_0&\mapsto [X]_0
\end{align*}‌ 
Assume that
\begin{equation}\label{seq1}
0\to X_{n+1}\to X_n\to\cdots \to X_{0}\to 0
\end{equation}
is an exact sequence with $n+2$ terms in $\mcc^0$.
Since $\mcc^0\subseteq\modd\Lambda$, the exact sequence \eqref{seq1} is in $\modd\Lambda$ and so $\sum^{n+1}_{i=0}(-1)^i[X_i]=0$ (see \cite[Proposition $\mathrm{VIII}$.4.1]{Bas}). Therefore, $i$ induces a well-defined group homomorphism
\begin{align*}
\alpha:\mathrm{K}_0(\mcc^0)&  \rightarrow \mathrm{K}_0(\modd\Lambda),\\
 [X]_{\mcc^0}&\mapsto [X]
\end{align*}
between Grothendieck groups.
It is enough to find a group homomorphism that is mutually inverse of $\alpha$.
Suppose that $X\in\modd\Lambda$. Since $\gd\Lambda\leq n$, there exists a projective resolution
\begin{equation}\label{pr}
0\to P_n^X\to P^X_{n-1}\to \cdots\to P_1^X \to P_0^X\to X\to0,
\end{equation}
of $X$.
We define the map
\begin{align*}
\mathrm{Ind}_{\mcc^0}:\modd\Lambda &\to\mathrm{K}_0(\mcc^0,0).\\
X &\mapsto \sum^{n}_{i=0}(-1)^i[P_i^X]_0
\end{align*}‌ 
$\mathrm{Ind}_{\mcc^0}(X)$ is independent of the choice of the projective resolution \eqref{pr} of $X$. In fact, if we consider the other projective resolution 
\begin{align*}
0\to Q_n^X\to Q^X_{n-1}\to \cdots\to Q_1^X \to Q_0^X\to X\to0,
\end{align*}
of $X$, then by Schanuel’s lemma (for example see \cite[Exercise 3.15]{Ro}), we have 
$$P_n^X\oplus Q_{n-1}^X\oplus P_{n-2}^X\oplus\cdots\cong Q_n^X\oplus P_{n-1}^X\oplus Q_{n-2}^X\oplus\cdots$$
Hence,  $\sum^{n}_{i=0}(-1)^i[P^X_i]_0=\sum^{n}_{i=0}(-1)^i[Q^X_i]_0$ and so $\mathrm{Ind}_{\mcc^0}$ is well-defined. It is clear that one can be extended $\mathrm{Ind}_{\mcc^0}$ to the group homomorphism 
$$\mathrm{Ind}_{\mcc^0}:\mathrm{K}_0(\modd\Lambda,0) \rightarrow\mathrm{K}_0(\mcc^0,0).$$
Now, assume that $0\to X\to Y\to Z\to 0$ is a short exact sequence in $\modd\Lambda$. 
Consider the diagram
\begin{center}
\scalebox{.8}{
\begin{tikzpicture}
\node (X00) at (-4,0) {$0$};
\node (X0) at (-2,0) {$P_n^X$};
\node (X1) at (0,0) {$P_{n-1}^X$};
\node (X2) at (2,0) {$\cdots$};
\node (X3) at (4,0) {$P_0^X$};
\node (X4) at (6,0) {$X$};
\node (X5) at (8,0) {$0$};
\node (X04) at (-4,-4) {$0$};
\node (X05) at (-2,-4) {$P_n^Z$};
\node (X6) at (0,-4) {$P_{n-1}^Z$};
\node (X7) at (2,-4) {$\cdots$};
\node (X8) at (4,-4) {$P_0^Z$};
\node (X9) at (6,-4) {$Z$};
\node (X10) at (8,-4) {$0$};
\node (X11) at (6,2) {$0$};
\node (X12) at (6,-6) {$0$};
\node (X13) at (6,-2) {$Y$};
\draw [->,thick] (X00) -- (X0);
\draw [->,thick] (X0) -- (X1);
\draw [->,thick] (X1) -- (X2);
\draw [->,thick] (X2) -- (X3);
\draw [->,thick] (X3) -- (X4);
\draw [->,thick] (X4) -- (X5);
\draw [->,thick] (X04) -- (X05);
\draw [->,thick] (X05) -- (X6);
\draw [->,thick] (X6) -- (X7);
\draw [->,thick] (X7) -- (X8);
\draw [->,thick] (X8) -- (X9);
\draw [->,thick] (X9) -- (X10);
\draw [->,thick] (X11) -- (X4);
\draw [->,thick] (X9) -- (X12);
\draw [->,thick] (X4) -- (X13);
\draw [->,thick] (X13) -- (X9);
\end{tikzpicture}}
\end{center}
where the rows are projective resolutions. Then, by Horseshoe lemma (for example see \cite[Proposition 6.24]{Ro}), we can complete the above diagram to the  commutative diagram
\begin{center}
\scalebox{.8}{
\begin{tikzpicture}
\node (X00) at (-4,0) {$0$};
\node (X0) at (-2,0) {$P_n^X$};
\node (X1) at (0,0) {$P_{n-1}^X$};
\node (X2) at (2,0) {$\cdots$};
\node (X3) at (4,0) {$P_0^X$};
\node (X4) at (6,0) {$X$};
\node (X5) at (8,0) {$0$};
\node (X04) at (-4,-4) {$0$};
\node (X05) at (-2,-4) {$P_n^Z$};
\node (X6) at (0,-4) {$P_{n-1}^Z$};
\node (X7) at (2,-4) {$\cdots$};
\node (X8) at (4,-4) {$P_0^Z$};
\node (X9) at (6,-4) {$Z$};
\node (X10) at (8,-4) {$0$};
\node (X11) at (6,2) {$0$};
\node (X02) at (0,2) {$0$};
\node (X01) at (-2,2) {$0$};
\node (X17) at (4,2) {$0$};
\node (X18) at (2,2) {$0$};
\node (X08) at (0,-6) {$0$};
\node (X09) at (-2,-6) {$0$};
\node (X12) at (6,-6) {$0$};
\node (X19) at (4,-6) {$0$};
\node (X20) at (2,-6) {$0$};
\node (X13) at (6,-2) {$Y$};
\node (X21) at (8,-2) {$0$};
\node (X012) at (-4,-2) {$0$};
\node (X013) at (-2,-2) {$P_n^Y$};
\node (X14) at (0,-2) {$P_{n-1}^Y$};
\node (X15) at (2,-2) {$\cdots$};
\node (X16) at (4,-2) {$P_0^Y$};
\draw [->,thick] (X00) -- (X0);
\draw [->,thick] (X0) -- (X1);
\draw [->,thick] (X1) -- (X2);
\draw [->,thick] (X2) -- (X3);
\draw [->,thick] (X3) -- (X4);
\draw [->,thick] (X4) -- (X5);
\draw [->,thick] (X04) -- (X05);
\draw [->,thick] (X05) -- (X6);
\draw [->,thick] (X6) -- (X7);
\draw [->,thick] (X7) -- (X8);
\draw [->,thick] (X8) -- (X9);
\draw [->,thick] (X9) -- (X10);
\draw [->,thick] (X11) -- (X4);
\draw [->,thick] (X9) -- (X12);
\draw [->,thick] (X4) -- (X13);
\draw [->,thick] (X13) -- (X9);
\draw [->,thick] (X012) -- (X013);
\draw [->,thick] (X013) -- (X14);
\draw [->,thick] (X14) -- (X15);
\draw [->,thick] (X15) -- (X16);
\draw [->,thick] (X16) -- (X13);
\draw [->,thick] (X2) -- (X15);
\draw [->,thick] (X15) -- (X7);
\draw [->,thick] (X3) -- (X16);
\draw [->,thick] (X16) -- (X8);
\draw [->,thick] (X17) -- (X3);
\draw [->,thick] (X18) -- (X2);
\draw [->,thick] (X8) -- (X19);
\draw [->,thick] (X7) -- (X20);
\draw [->,thick] (X13) -- (X21);
\draw [->,thick] (X02) -- (X1);
\draw [->,thick] (X01) -- (X0);
\draw [->,thick] (X6) -- (X08);
\draw [->,thick] (X05) -- (X09);
\draw [->,thick] (X1) -- (X14);
\draw [->,thick] (X14) -- (X6);
\draw [->,thick] (X0) -- (X013);
\draw [->,thick] (X013) -- (X05);
\end{tikzpicture}}
\end{center}
where all columns are exact, the middle row is a projective resolution of $Y$ such that $P^Y_i\cong P^X_i\oplus P^Z_i$, for any $i\in\{0,\dots,n\}$. Therefore, we have
$$\mathrm{Ind}_{\mcc^0}([X]_0-[Y]_0+[Z]_0)=\sum^{n}_{i=0}(-1)^i[P_i^X]_0-\sum^{n}_{i=0}(-1)^i[P_i^X\oplus P_i^Z]_0+\sum^{n}_{i=0}(-1)^i[P_i^Z]_0=0,$$
and so $\mathrm{Ind}_{\mcc^0}$ induces the well-defined group homomorphism
\begin{align*}
\beta:\mathrm{K}_0(\modd\Lambda)& \rightarrow \mathrm{K}_0(\mcc^0),\\
 [X] &\mapsto \sum^{n}_{i=0}(-1)^i[P_i^X]_{\mcc^0}
\end{align*}
between Grothendieck groups.
It remains to show that $\alpha$ and $\beta$ are mutually inverse.
Consider $X\in\modd\Lambda$ and the projective resolution \eqref{pr} of $X$.
We know that any projective resolution is an exact sequence in $\modd\Lambda$ and clearly $[X]=\sum^{n}_{i=0}(-1)^i[P^X_i]$. Therefore, we have
$$\alpha\circ\beta([X])=\alpha(\sum^{n}_{i=0}(-1)^i[P_i^X]_{\mcc^0})=\sum^{n}_{i=0}(-1)^i[P_i^X]=[X].$$
Consider $X\in\mcc^0$ and the projective resolution \eqref{pr} of $X$.
We know that any projective resolution is an exact sequence with $n+2$ terms in $\mcc^0$ and clearly $[X]_{\mcc^0}=\sum^{n}_{i=0}(-1)^i[P^X_i]_{\mcc^0}$. Therefore, we have
$$\beta\circ\alpha([X]_{\mcc^0})=\beta([X])=\sum^{n}_{i=0}(-1)^i[P^X_i]_{\mcc^0}=[X]_{\mcc^0}.$$
\end{proof}
As we mentioned before, by \cite[Theorem $\mathrm{III}$.3.5]{ASS}, for a basic finite dimensional $K$-algebra $\Lambda$ with a complete set $\{S_1,\dots,S_m\}$ of the isomorphism classes of simple $\Lambda$-modules, $\{[S_1],\dots,[S_m]\}$ is a basis for $\mathrm{K}_0(\modd\Lambda)$ and $\mathrm{K}_0(\modd\Lambda)\cong\mathbb{Z}^m$. As a consequence of this fact and Theorem \ref{T1}, we have the following corollary.
\begin{corollary}\label{Zm}
Let $\Lambda$ be an $n$-hereditary algebra and $\{S_1,\dots,S_m\}$ 
be a complete set of the isomorphism classes of simple $\Lambda$-modules. Then, $\mathrm{K}_0(\mcc^0)\cong\mathbb{Z}^m$.
\end{corollary}
\subsection{$\boldsymbol{n}$-hereditary algebras with  the positive definite quadratic form}
 We recall that an arbitrary quadratic form $q$ is \textit{positive definite} if $q(x)>0$ for all $x\neq0$. For hereditary algebras, one can use the homological quadratic forms to determine the representation type of algebras. Indeed,  it is a well-known fact that the homological quadratic form of an hereditary algebra is positive definite if and only if that algebra is representation finite.
 It is natural to ask if one can generalize this fact to $n$-hereditary algebras. For this purpose, we use the subcategory $\mcc^0$ of $\modd\Lambda$ instead of $\modd\Lambda$. Any exact sequence with $n+2$ terms in $\mcc^0$ is an exact sequence in $\modd\Lambda$. So we can restrict the bilinear form $B$ and the corresponding homological quadratic form $\chi$ to $\mathrm{K}_0(\mcc^0)$. Moreover, by Theorem \ref{T1}, we can define $B([X]_{\mcc^0},[Y]_{\mcc^0})\coloneqq B([X],[Y])$ and $\chi([X]_{\mcc^0})\coloneqq\chi([X])$, for any $X,Y\in\mcc^0$.

In the following theorem, we give equivalent conditions to $n$-representation finiteness of $n$-hereditary algebras.
\begin{theorem}\label{form0}
Let $\Lambda$ be an $n$-hereditary algebra with the complete set $\{P_1,\dots,P_m\}$ and $\{I_1,\dots,I_m\}$ of the isomorphism classes of indecomposable projective and indecomposable injective $\Lambda$-modules, respectively. Then the following are equivalent.
\begin{enumerate}
\item[(a)]
$\Lambda$ is $n$-representation finite.
\item[(b)]
$\{X\in\ind\mathscr{P}\,|\,\chi([X])=l\}$ is a finite set for any positive integer $l$.
\item[(c)]
$\{X\in\ind\mathscr{P}\,|\,\chi([X])=\chi([P_i])\}$ is a finite set for any $i\in\{1,\dots,m\}$.
\item[(d)]
$\{X\in\ind\mathscr{I}\,|\,\chi([X])=l\}$ is a finite set for any positive integer $l$.
\item[(e)]
$\{X\in\ind\mathscr{I}\,|\,\chi([X])=\chi([I_i])\}$ is a finite set for any $i\in\{1,\dots,m\}$.
\item[(f)]
$\{X\in\ind\mcc^0\,|\,\chi([X])=l\}$ is a finite set for any positive integer $l$.
\end{enumerate}
\end{theorem}
\begin{proof}
We only prove the equivalence of $(\mathrm{a})$, $(\mathrm{b})$ and $(\mathrm{c})$. The proof of the other parts are similar. Assume that $\Lambda$ is $n$-representation finite, then Theorem \ref{th1} implies that $\mcc ^0=\mathscr{P}=\mathscr{I}$ and by Proposition \ref{per}, $\ind\mcc ^0$ is a finite set. Therefore,  the part $(\mathrm{b})$ holds. By Remark \ref{fo}, $\chi([P_i])=\dim_K\End_\Lambda(P_i) >0$ for any $i\in\{1,\dots,m\}$ and so $(\mathrm{c})$ follows from $(\mathrm{b})$. Suppose that $\{X\in\ind\mathscr{P}\,|\,\chi([X])=\chi([P_i])\}$ is a finite set, for any $i\in\{1,\dots,m\}$. 
If $\Lambda$ is not $n$-representation finite, then Theorem \ref{de} implies that it is $n$-representation infinite. By Proposition \ref{Thio}$(\mathrm{a})$, for any indecomposable projective $\Lambda$-module $P_i$ and any integer $j\geq0$, $\nu_n^{-j}P_i\cong\tau_n^{-j}P_i$ is an indecomposable $\Lambda$-module in $\mathscr{P}$ and these modules are non-isomorphic.
By Remark \ref{fo}, we have
$$\chi([\tau_n^{-j}P_i])=\dim_K\End_\Lambda(\tau_n^{-j}P_i)-\dim_K\Ext_\Lambda^n(\tau_n^{-j}P_i,\tau_n^{-j}P_i).$$  
Proposition \ref{n-ext}$(\mathrm{a})$ implies that
$$\dim_K\End_\Lambda(\tau_n^{-j}P_i)=\dim_K\End_\Lambda(P_i)\,\,\,\quad\,\text{and}\,\,\,\quad\Ext_\Lambda^n(\tau_n^{-j}P_i,\tau_n^{-j}P_i)=0.$$

Therefore, for any $j\geq0$ we have 
$\chi([\tau_n^{-j}P_i])=\dim_K\End_\Lambda(P_i)=\chi([P_i])$. Then $\{X\in\ind\mathscr{P}\,|\,\chi([X])=\chi([P_i])\}$ is an infinite set for any $i\in\{1,\dots,m\}$ and this is a contradiction.
\end{proof}
If the Gabriel quiver of $\Lambda$ is acyclic, then $\dim_K\End_\Lambda(P_i)=\dim_K\End_\Lambda(I_i)=1$ for any $i\in\{1,\dots,m\}$. Therefore, we have the following  immediate result. 
\begin{corollary}\label{form00}
		Let $\Lambda$ be an $n$-hereditary algebra with the acyclic Gabriel quiver. Then the following are equivalent.
		\begin{enumerate}
			\item[(a)]
			$\Lambda$ is $n$-representation finite.
			\item[(b)]
			$\{X\in\ind\mathscr{P}\,|\,\chi([X])=1\}$ is a finite set.
			\item[(c)]
			$\{X\in\ind\mathscr{I}\,|\,\chi([X])=1\}$ is a finite set.
			\item[(d)]
		 	$\{X\in\ind\mcc^0\,|\,\chi([X])=1\}$ is a finite set.
		\end{enumerate}
\end{corollary}
 We recall that an element $x\in\mathrm{K}_0(\modd\Lambda)$ with $\chi(x) = 1$ is called a \textit{root} of $\chi$. A root $x$ is \textit{positive} if there exists an $X\in\modd\Lambda$ such that $x = [X]$.
The above corollary shows that an $n$-hereditary algebra $\Lambda$ with the acyclic Gabriel quiver is $n$-representation finite if and only if the number of isomorphism classes of indecomposable $\Lambda$-modules $X\in \mcc^0$ that $[X]$ is a positive root of $\chi$ is finite. Note that when the Gabriel quiver of an $n$-hereditary algebra $\Lambda$ is acyclic, Remark \ref{fo} and Proposition \ref{n-ext} imply that for any indecomposable $\Lambda$-modules $X$ with $X\in \mathscr{P}\vee\mathscr{I}$, $[X]$ is a positive root of $\chi$.

Now, we give a sufficient condition for the $n$-representation finiteness of $n$-hereditary algebras by using quadratic forms.
\begin{theorem}\label{form}
Let $\Lambda$ be an $n$-hereditary algebra with odd $n$. If the restriction of the homological quadratic form $\chi$ to $\mathrm{K}_0(\mcc^0)$ is positive definite, then $\Lambda$ is $n$-representation finite.
\end{theorem}
\begin{proof}
Assume that $\Lambda$ is not $n$-representation finite. Theorem \ref{form0} implies that there exists a positive integer $l$ such that $\{X\in\ind\mathscr{P}\,|\,\chi([X])=l\}$ is an infinite set.
In fact, the quadratic form $\chi$ corresponds infinitely many different elements in $\mathrm{K}_0(\mcc^0)$ to $l$. By Corollary \ref{Zm}, $\mathrm{K}_0(\mcc^0)\cong\mathbb{Z}^m$. Therefore, it can be considered as $\mathbb{Z}\otimes_\mathbb{Z}\mathrm{K}_0(\mcc^0)$. By the assumption, $\chi$ is positive definite and so it can be extended  to a positive definite quadratic form on $\mathbb{Q}\otimes_\mathbb{Z}\mathrm{K}_0(\mcc^0)$, by linearity. Also, $\chi$ can be extended  to a continuous quadratic form on $\mathbb{R}\otimes_\mathbb{Z}\mathrm{K}_0(\mcc^0)$. We claim that $\chi$ is positive definite on $\mathbb{R}\otimes_\mathbb{Z}\mathrm{K}_0(\mcc^0)$. The associated symmetric bilinear form to $B$ on
$\mathbb{Q}\otimes_\mathbb{Z}\mathrm{K}_0(\mcc^0)$ is defined as
$$B^\prime(x,y)=1/2\big(B(x,y)+B(y,x)\big).$$ 
By applying the Gram-Schmidt process, there is an orthogonal basis $\{v_1,\dots,v_m\}$ of
$\mathbb{Q}\otimes_\mathbb{Z}\mathrm{K}_0(\mcc^0)$ relative to $B^\prime$. Consider $0\neq v=\sum_{i=1}^m\alpha^iv_i\in\mathbb{R}\otimes_\mathbb{Z}\mathrm{K}_0(\mcc^0)$. 
We have, $\chi(v)=\sum_{i=1}^m\alpha^2_iB^\prime(v_i,v_i)>0$ and so our claim is proved. We know that for a
positive definite quadratic form $\chi$ on a finite dimensional $\mathbb{R}$-vector space
$V$, the set $\{x\in V\,|\, \chi(x)\leq l\}$ is bounded for
each $l\in\mathbb{R}$. If we apply this fact for the $\mathbb{R}$-vector space $\mathbb{R}\otimes_\mathbb{Z}\mathrm{K}_0(\mcc^0)\cong \mathbb{R}^m$ and the positive definite quadratic form $\chi$, then we obtain the set $\{x\in \mathbb{R}^m\,|\, \chi(x)\leq l\}$ is bounded. Consequently, the set $\{x\in\mathbb{Z}^m\,|\, \chi(x)\leq l\}$ is finite. 
Namely, there exist finitely many dimension vectors $x$ such that $\chi(x)\leq l$. Therefore, Theorem \ref{dv}$(\mathrm{a})$ implies that the set
$$\{x\in\mathrm{K}_0(\mcc^0)\,|\, \chi(x)\leq l\}$$ 
is finite. But this is a contradiction and so $\Lambda$ is $n$-representation finite.
\end{proof}
The following result immediately follows from the above theorem and Theorem \ref{T1}.
\begin{corollary}\label{cor1}
Let $\Lambda$ be an $n$-hereditary algebra with odd $n$. If $\chi$ is positive definite, then $\Lambda$ is $n$-representation finite.
\end{corollary}
In the following examples, we show that the converse of Theorem \ref{form} is not satisfied.

\begin{example}\label{exam}
Let $\Lambda$ be the algebra given by the quiver
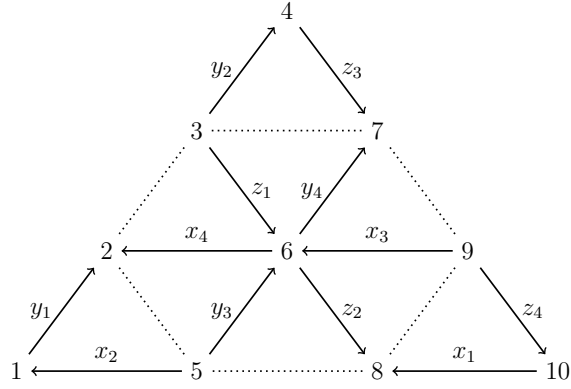
\begin{center}
	\scalebox{.8}{
		\begin{tikzpicture}
			\node (X1) at (-1.5,-2) {$1$};
			\node (X2) at (0,0) {$2$};
			\node (X3) at (1.5,2) {$3$};
			\node (X4) at (3,4) {$4$};
			\node (X5) at (1.5,-2) {$5$};
			\node (X6) at (3,0) {$6$};
			\node (X7) at (4.5,2) {$7$};
			\node (X8) at (4.5,-2) {$8$};
			\node (X9) at (6,0) {$9$};
			\node (X10) at (7.5,-2) {$10$};
			\draw [->,thick] (X1) -- (X2)node [midway,left] {$y_1$};
			\draw [->,thick] (X3) -- (X4)node [midway,left] {$y_2$};
			\draw [->,thick] (X5) -- (X6)node [midway,left] {$y_3$};
			\draw [->,thick] (X6) -- (X7)node [midway,left] {$y_4$};
			\draw [->,thick] (X3) -- (X6)node [midway,right] {$z_1$};
			\draw [->,thick] (X6) -- (X8)node [midway,right] {$z_2$};
			\draw [->,thick] (X4) -- (X7)node [midway,right] {$z_3$};
			\draw [->,thick] (X9) -- (X10)node [midway,right] {$z_4$};
			\draw [->,thick] (X5) -- (X1)node [midway,above] {$x_2$};
			\draw [->,thick] (X10) -- (X8)node [midway,above] {$x_1$};
			\draw [->,thick] (X9) -- (X6)node [midway,above] {$x_3$};
			\draw [->,thick] (X6) -- (X2)node [midway,above] {$x_4$};
			\draw [-,thick,dotted] (X2) -- (X3);
			\draw [-,thick,dotted] (X3) -- (X7);
			\draw [-,thick,dotted] (X5) -- (X2);
			\draw [-,thick,dotted] (X8) -- (X5);
			\draw [-,thick,dotted] (X9) -- (X8);
			\draw [-,thick,dotted] (X9) -- (X7);
	\end{tikzpicture}}
\end{center}
bound by relations $x_2y_1-y_3x_4=z_1y_4-y_2z_3=z_4x_1-x_3z_2=y_3z_2= x_3y_4=z_1x_4=0$. $\Lambda$ is $2$-representation finite (see \cite[Theorem 1.18]{I4}).
Since the following quiver with Euclidean underlying graph $\tilde{E}_6$ is a subquiver of $Q$, $\Lambda$ is representation infinite. 
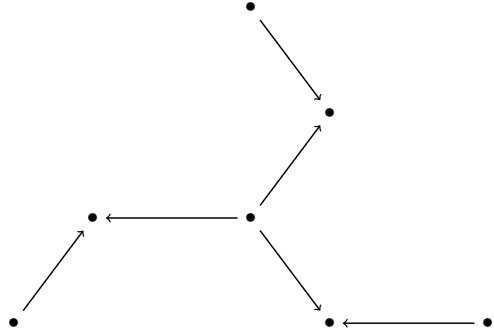
\begin{center}
	\scalebox{.7}{
		\begin{tikzpicture}
			\node (X1) at (-1.5,-2) {$\bullet$};
			\node (X2) at (0,0) {$\bullet$};
			\node (X4) at (3,4) {$\bullet$};
			\node (X6) at (3,0) {$\bullet$};
			\node (X7) at (4.5,2) {$\bullet$};
			\node (X8) at (4.5,-2) {$\bullet$};
			\node (X10) at (7.5,-2) {$\bullet$};
			\draw [->,thick] (X1) -- (X2);
			\draw [->,thick] (X6) -- (X7);
			\draw [->,thick] (X6) -- (X8);
			\draw [->,thick] (X4) -- (X7);
			\draw [->,thick] (X10) -- (X8);
			\draw [->,thick] (X6) -- (X2);
	\end{tikzpicture}}
\end{center}
Consider the Cartan matrix of $\Lambda$ 
\begin{align*}
C_\Lambda =\begin{bmatrix}                                                                                     
1&0&0&0&1&0&0&0&0&0\\
1&1&0&0&1&1&0&0&1&0\\
0&0&1&0&0&0&0&0&0&0\\
0&0&1&1&0&0&0&0&0&0\\
0&0&0&0&1&0&0&0&0&0\\
0&0&1&0&1&1&0&0&1&0\\
0&0&1&1&1&1&1&0&0&0\\
0&0&1&0&0&1&0&1&1&1\\
0&0&0&0&0&0&0&0&1&0\\
0&0&0&0&0&0&0&0&1&1
\end{bmatrix}.
\end{align*}
Suppose that $X\in\modd\Lambda$ with $\underline{\dim}X=(1 2 0 1 0 3 2 2 0 1)$. By Proposition \ref{th2}$(\mathrm{b})$, we have
$$\chi([X])=q(\underline{\dim}X)=(\underline{\dim}X)^t(C_\Lambda^{-1})^t \underline{\dim}X=0.$$
Consider the following projective resolution of $X$.
\begin{align*}
0\to P_n^X\to \cdots\to P_1^X \to P_0^X\to X\to0.
\end{align*}
We have  $[X]=\sum^{n}_{i=0}(-1)^i[P^X_i]$. Set $x=\sum^{n}_{i=0}(-1)^i[P^X_i]_{\mcc^0}$. We have  
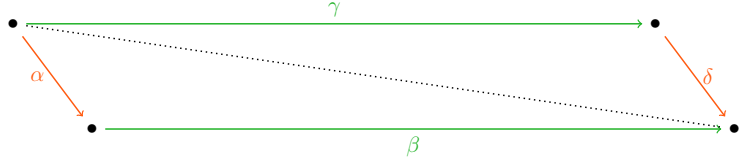
$$\chi(x)=\chi(\sum^{n}_{i=0}(-1)^i[P^X_i]_{\mcc^0})=\chi(\sum^{n}_{i=0}(-1)^i[P^X_i])=\chi([X])=0,$$ 
and so there is an $x\in\mathrm{K}_0(\mcc^0)$ such that $\chi(x)=0$. Therefore, $\chi$ is not positive definite.
\end{example}
Let $k$ be a natural number. An $n$-representation finite algebra $\Lambda$ is called $k$-homogeneous if for any indecomposable projective $\Lambda$-module $P$, $\tau_n^{-(k-1)}(P)$ is an indecomposable injective $\Lambda$-module (see \cite[Definition 1.2]{HI1}).
We recall that if $\Lambda_i$ is an $k$-homogeneous $n_i$-representation finite algebra for any $i\in\{1,\dots,t\}$, then $\Lambda_1\otimes_K\cdots\otimes_K\Lambda_t\,$ is an $k$-homogeneous $(n_1+\cdots+n_t)$-representation finite algebra (see \cite[Corollary 1.5]{HI1}).
\begin{example}\label{exam2}
Consider the following quivers.
 $$
\begin{array}{cc}
	\scalebox{.7}{\begin{tikzpicture}[color=gree]
			\node (X1) at (-1.5,-2) {$\bullet$};
			\node (X2) at (0,0) {$\bullet$};
			\node (X3) at (1.5,2) {$\bullet$};
			\node (X4) at (3,4) {$\bullet$};
			\node (X5) at (1.5,-2) {$\bullet$};
			\node (X6) at (3,0) {$\bullet$};
			\node (X7) at (4.5,2) {$\bullet$};
			\node (X8) at (4.5,-2) {$\bullet$};
			\node (X9) at (6,0) {$\bullet$};
			\node (X10) at (7.5,-2) {$\bullet$};
			\node (X0) at (-2.5,-2) {$Q:\,\,$};
			\draw [->,thick] (X1) -- (X2);
			\draw [->,thick] (X3) -- (X4);
			\draw [->,thick] (X5) -- (X6);
			\draw [->,thick] (X6) -- (X7);
			\draw [->,thick] (X3) -- (X6);
			\draw [->,thick] (X6) -- (X8);
			\draw [->,thick] (X4) -- (X7);
			\draw [->,thick] (X9) -- (X10);
			\draw [->,thick] (X5) -- (X1);
			\draw [->,thick] (X10) -- (X8);
			\draw [->,thick] (X9) -- (X6);
			\draw [->,thick] (X6) -- (X2);
			\draw [-,thick,dotted] (X2) -- (X3);
			\draw [-,thick,dotted] (X3) -- (X7);
			\draw [-,thick,dotted] (X5) -- (X2);
			\draw [-,thick,dotted] (X8) -- (X5);
			\draw [-,thick,dotted] (X9) -- (X8);
			\draw [-,thick,dotted] (X9) -- (X7);
	\end{tikzpicture}} & 
	\scalebox{.7}{\begin{tikzpicture}[color=Gr]
			\node (X0) at (-1,0) {$\,\,Q^\prime:\,\,$};
			\node (X1) at (0,0) {$\bullet$};
			\node (X2) at (3,0) {$\bullet$};
			\node (X3) at (6,0) {$\bullet$};
			\draw [->,thick] (X1) -- (X2);
			\draw [->,thick] (X3) -- (X2);
	\end{tikzpicture}}
\end{array}
$$ 
Set $\Lambda_2=KQ^\prime$ and $\Lambda_1=KQ/\mathcal{I}$ such that  $\mathcal{I}$ is the ideal generated by the dotted relation ($\Lambda_1$ is the algebra in the example \ref{exam}).
 $\Lambda_1$ is $2$-homogeneous $2$-representation finite and $\Lambda_2$ is $2$-homogeneous $1$-representation finite (see \cite{HI1}). Therefore, $\Lambda\coloneqq\Lambda_1\otimes_K \Lambda_2$ is $2$-homogeneous $3$-representation finite and clearly $\mcc^0=\add(\Lambda\oplus D\Lambda)$ is the unique $3$-cluster tilting subcategory of $\modd\Lambda$.  
 By \cite[Lemma 1.3]{Le}, $\Lambda\cong K(Q\otimes Q^\prime)/ \mathcal{I}$
such that $Q\otimes Q^\prime$ is the quiver
\begin{center}
\scalebox{.55}{
	\begin{tikzpicture}
		\node (X1) at (0,-.4) {$1$};
		\node (X2) at (1.4,1.4) {$2$};
		\node (X3) at (2.9,3) {$3$};
		\node (X4) at (4.4,4.8){$4$};
		\node (X5) at (2.4,-.6) {$5$};
		\node (X6) at (4,1.2) {$6$};
		\node (X7) at (5.5,2.8){$7$};
		\node (X8) at (5,-.8){$8$};
		\node (X9) at (6.6,1) {$9$};
		\node (X10) at (7.6,-1){$10$};
		\node (X11) at (11.2,.4){$11$};
		\node (X12) at (12.6,2.2) {$12$};
		\node (X13) at (14.1,3.8) {$13$};
		\node (X14) at (15.6,5.6){$14$};
		\node (X15) at (13.6,.2) {$15$};
		\node (X16) at (15.2,2) {$16$};
		\node (X17) at (16.7,3.6){$17$};
		\node (X18) at (16.2,0){$18$};
		\node (X19) at (17.8,1.8){$19$};
		\node (X20) at (18.8,-.2){$20$};
		\node (X21) at (22.4,1.2){$21$};
		\node (X22) at (23.8,3){$22$};
		\node (X23) at (25.3,4.6) {$23$};
		\node (X24) at (26.8,6.4){$24$};
		\node (X25) at (24.8,1){$25$};
		\node (X26) at (26.4,2.8){$26$};
		\node (X27) at (27.9,4.4){$27$};
		\node (X28) at (27.4,0.8){$28$};
		\node (X29) at (29,2.6){$29$};
		\node (X30) at (30,.6){$30$};
		\draw [->,thick,color=Gr] (X1) -- (X11);
		\draw [->,thick,color=Gr] (X21) -- (X11);
		\draw [->,thick,color=Gr] (X2) -- (X12);
		\draw [->,thick,color=Gr] (X22) -- (X12);
		\draw [->,thick,color=Gr] (X3) -- (X13);
		\draw [->,thick,color=Gr] (X23) -- (X13);
		\draw [->,thick,color=Gr] (X4) -- (X14);
		\draw [->,thick,color=Gr] (X24) -- (X14);
		\draw [->,thick,color=Gr] (X5) -- (X15);
		\draw [->,thick,color=Gr] (X25) -- (X15);
		\draw [->,thick,color=Gr] (X6) -- (X16);
		\draw [->,thick,color=Gr] (X26) -- (X16);
		\draw [->,thick,color=Gr] (X7) -- (X17);
		\draw [->,thick,color=Gr] (X27) -- (X17);
		\draw [->,thick,color=Gr] (X8) -- (X18);
		\draw [->,thick,color=Gr] (X28) -- (X18);
		\draw [->,thick,color=Gr] (X9) -- (X19);
		\draw [->,thick,color=Gr] (X29) -- (X19);
		\draw [->,thick,color=Gr] (X10) -- (X20);
		\draw [->,thick,color=Gr] (X30) -- (X20);
		\draw [->,thick,color=gree] (X21) -- (X22);
		\draw [->,thick,color=gree] (X23) -- (X24);
		\draw [->,thick,color=gree] (X25) -- (X26);
		\draw [->,thick,color=gree] (X26) -- (X27);
		\draw [->,thick,color=gree] (X23) -- (X26);
		\draw [->,thick,color=gree] (X26) -- (X28);
		\draw [->,thick,color=gree] (X24) -- (X27);
		\draw [->,thick,color=gree] (X29) -- (X30);
		\draw [->,thick,color=gree] (X25) -- (X21);
		\draw [->,thick,color=gree] (X30) -- (X28);
		\draw [->,thick,color=gree] (X29) -- (X26);
		\draw [->,thick,color=gree] (X26) -- (X22);
		\draw [-,thick,dotted,color=gree] (X22) -- (X23);
		\draw [-,thick,dotted,color=gree] (X23) -- (X27);
		\draw [-,thick,dotted,color=gree] (X25) -- (X22);
		\draw [-,thick,dotted,color=gree] (X28) -- (X25);
		\draw [-,thick,dotted,color=gree] (X29) -- (X28);
		\draw [-,thick,dotted,color=gree] (X29) -- (X27);
		\draw [->,thick,color=gree] (X11) -- (X12);
		\draw [->,thick,color=gree] (X13) -- (X14);
		\draw [->,thick,color=gree] (X15) -- (X16);
		\draw [->,thick,color=gree] (X16) -- (X17);
		\draw [->,thick,color=gree] (X13) -- (X16);
		\draw [->,thick,color=gree] (X16) -- (X18);
		\draw [->,thick,color=gree] (X14) -- (X17);
		\draw [->,thick,color=gree] (X19) -- (X20);
		\draw [->,thick,color=gree] (X15) -- (X11);
		\draw [->,thick,color=gree] (X20) -- (X18);
		\draw [->,thick,color=gree] (X19) -- (X16);
		\draw [->,thick,color=gree] (X16) -- (X12);
		\draw [-,thick,dotted,color=gree] (X12) -- (X13);
		\draw [-,thick,dotted,color=gree] (X13) -- (X17);
		\draw [-,thick,dotted,color=gree] (X15) -- (X12);
		\draw [-,thick,dotted,color=gree] (X18) -- (X15);
		\draw [-,thick,dotted,color=gree] (X19) -- (X18);
		\draw [-,thick,dotted,color=gree] (X19) -- (X17);
		\draw [->,thick,color=gree] (X1) -- (X2);
		\draw [->,thick,color=gree] (X3) -- (X4);
		\draw [->,thick,color=gree] (X5) -- (X6);
		\draw [->,thick,color=gree] (X6) -- (X7);
		\draw [->,thick,color=gree] (X3) -- (X6);
		\draw [->,thick,color=gree] (X6) -- (X8);
		\draw [->,thick,color=gree] (X4) -- (X7);
		\draw [->,thick,color=gree] (X9) -- (X10);
		\draw [->,thick,color=gree] (X5) -- (X1);
		\draw [->,thick,color=gree] (X10) -- (X8);
		\draw [->,thick,color=gree] (X9) -- (X6);
		\draw [->,thick,color=gree] (X6) -- (X2);
		\draw [-,thick,dotted,color=gree] (X2) -- (X3);
		\draw [-,thick,dotted,color=gree] (X3) -- (X7);
		\draw [-,thick,dotted,color=gree] (X5) -- (X2);
		\draw [-,thick,dotted,color=gree] (X8) -- (X5);
		\draw [-,thick,dotted,color=gree] (X9) -- (X8);
		\draw [-,thick,dotted,color=gree] (X9) -- (X7);
\end{tikzpicture}}
\end{center}
and $\mathcal{I}$ is the ideal generated by the following relations.
\begin{itemize}
\item
The dotted relations. In fact, these relations come from the quiver $Q$.
\item

The new relation $\alpha\beta-\gamma\delta$, for any two paths such $\alpha\beta$ and $\gamma\delta$ (note that $\alpha$ and $\delta$ are arrows of $Q$ and $\beta$ and $\gamma$ are arrows of $Q^\prime$).
\begin{center}
\scalebox{.7}{
\begin{tikzpicture}
\node (X4) at (3,4) {$\bullet$};
\node (X7) at (4.5,2) {$\bullet$};
\node (X14) at (15.2,4){$\bullet$};
\node (X17) at (16.7,2) {$\bullet$};
\draw [->,thick,color=gree] (X4) -- (X7)node [midway,left] {$\alpha$};
\draw [->,thick,color=gree] (X14) -- (X17)node [midway,right] {$\delta$};
\draw [->,thick,color=Gr] (X4) -- (X14)node [midway,above] {$\gamma$};
\draw [->,thick,color=Gr] (X7) -- (X17)node [midway,below] {$\beta$};
\draw [-,thick,dotted] (X4) -- (X17);
\end{tikzpicture}}
\end{center}
\end{itemize}

In Example \ref{exam}, we compute the Cartan matrix of $\Lambda_1=KQ/\mathcal{I}$. It is easy to see that the Cartan matrix $C_\Lambda$ is equal to the blocked matrix
\begin{align*}
 C_\Lambda=\begin{bmatrix}                                                                                     
C_{\Lambda_1} & 0 &0\\
C_{\Lambda_1}&C_{\Lambda_1} &C_{\Lambda_1}\\
0& 0 &C_{\Lambda_1}
\end{bmatrix}.
\end{align*}
Consider $X\in\modd\Lambda$ with the dimension vector $\underline{\dim}X=(1 2 0 1 0 3 2 2 01 0 0 0 \cdots 0)$. By Proposition \ref{th2}$(\rm{b})$, we have
$$\chi([X])=q(\underline{\dim}X)=(\underline{\dim}X)^t(C_\Lambda^{-1})^t \underline{\dim}X=0.$$
Assume that
\begin{align*}
0\to P_n^X\to \cdots\to P_1^X \to P_0^X\to X\to0
\end{align*}
is a projective resolution of $X$. We have  $[X]=\sum^{n}_{i=0}(-1)^i[P^X_i]$. Set $x=\sum^{n}_{i=0}(-1)^i[P^X_i]_{\mcc^0}$. We have  
$$\chi(x)=\chi(\sum^{n}_{i=0}(-1)^i[P^X_i]_{\mcc^0})=\chi(\sum^{n}_{i=0}(-1)^i[P^X_i])=\chi([X])=0,$$ 
and so there is an $x\in\mathrm{K}_0(\mcc^0)$ such that $\chi(x)=0$. Therefore, $\chi$ is not positive definite.
\end{example}

According to Examples \ref{exam} and \ref{exam2}, the converse of Corollary \ref{cor1} is not satisfied.

As we mentioned before, for a hereditary algebra $A$, the homological quadratic form $\chi$ is positive definite if and only if $A$ is representation finite (see \cite[Theorem $\mathrm{VIII}$.3.6]{ARS}). To prove the inverse direction in \cite[Theorem $\mathrm{VIII}$.3.6]{ARS}, the authors used the fact that for a hereditary algebra $A$, the homological quadratic form $\chi$ is positive definite if and only if $\chi([X]) > 0$ for all non-zero $X\in\modd A$ (see \cite[Lemma $\mathrm{VIII}$.3.3]{ARS}).
One can see that if the restriction of $\chi$ to the set $\{[X]\,|\, X\in\ind A\}$ is positive definite then $A$ is still representation finite. Although this fact is known to the expert, we give a proof for it in the following remark.
\begin{remark}\label{h}
		Let $A$ be a hereditary algebra.
		\begin{enumerate}
		\item 
		The restriction of $\chi$ to $\{[X]\,|\, X\in\ind A\}$ is positive definite if and only if $A$ is representation finite. For the proof, assume that the restriction $\chi$ to $\{[X]\,|\, X\in\ind A\}$ is positive definite. Suppose to the contrary
		that $A$ is not representation finite. By \cite[Proposition $\mathrm{VIII}$.3.6)]{ARS}, there exists an indecomposable $A$-module $X$ with $\Ext^1(X,X) \neq0$. If $X$ is a brick, i.e., $\End_A(X)\cong K$, then we have $\chi([X])=\dim_K\End_A(X)-\dim_K\Ext_A^1(X,X)\leq0$. But this contradicts the assumption. If $X$ is not a brick, By \cite[Theorem]{R}, there exists a brick $X^\prime$ such that $\Ext^1(X^\prime,X^\prime) \neq0$. Then similarly we have $\chi([X^\prime])\leq0$. But this also contradicts the assumption. Therefore, $A$ is representation finite. The proof of the other side is clear.
		\item 
			By the preceding discussion, the following are equivalent.
		\begin{itemize}
			\item[(a)]
			$A$ is representation finite.
			\item[(b)]
			The homological quadratic form $\chi$ is positive definite on $\rm{K}_0(\modd A)$.
			\item[(c)]
			The restriction of $\chi$ to $\{[X]\,|\,0\neq X\in\modd A\}$ is positive definite.
			\item[(d)]
			The restriction of $\chi$ to $\{[X]\,|\, X\in\ind A\}$ is positive definite.
		\end{itemize}
		\end{enumerate}
\end{remark}
In Theorem \ref{form}, we gave the higher dimensional version of $(\rm{b})\Rightarrow (\rm{a})$ for $n$-hereditary algebras with odd $n$.
On the other hand, we gave the counterexamples for the higher dimensional analogue of $(\rm{a})\Rightarrow (\rm{b})$ for any $n$-hereditary algebra (see examples \ref{exam} and \ref{exam2}). So unlike the classical case, the higher dimensional analogue of the implication $(\rm{a})\Leftrightarrow(\rm{b})$ is not satisfied for $n$-hereditary algebras. Therefore, it is natural to ask what is the necessary and sufficient condition for the $n$-representation finiteness of $n$-hereditary algebras. To finding an answer for this question, we investigate the higher dimensional analogue of the other parts of Remark \ref{h}$(2)$.
\begin{proposition}
	Let $\Lambda$ be an $n$-hereditary algebra.
	\begin{itemize}
		\item[(a)]
		The restriction of $\chi$ to $\{[X]\,|\, X\in\ind(\mathscr{P}\vee\mathscr{I})\}$ is positive definite.
		\item[(b)]
		If $n$ is even, then the restriction of $\chi$ to $\{[X]\,|\,0\neq X\in\mathscr{P}\vee\mathscr{I}\}$ is positive definite. 	
	\end{itemize}
\end{proposition}
\begin{proof}
	$(\mathrm{a})$ Assume that $X\in\ind(\mathscr{P}\vee\mathscr{I})$. By Theorem \ref{n-ext}, $\Ext_\Lambda^n(X,X)=0$ and so by Remark \ref{fo} we have
	\begin{align*}
		\chi([X])&=\dim_K\Hom_\Lambda(X,X)>0.
	\end{align*}
	$(\mathrm{b})$ By Remark \ref{fo}, we have
	\begin{align*}
		\chi([X])=\dim_K\Hom_\Lambda(X,X)+\Ext_\Lambda^n(X,X)>0.
	\end{align*}	
\end{proof}
Let $\Lambda$ be an $n$-representation finite algebra. Theorem \ref{th1} implies that $\mcc ^0=\mathscr{P}=\mathscr{I}$. Therefore, the following result immediately follows from the above proposition.
\begin{corollary}
Let $\Lambda$ be an $n$-representation finite algebra.
\begin{itemize}
	\item[(a)]
The restriction of $\chi$ to $\{[X]\,|\, X\in\ind\mcc^0\}$ is positive definite.
	\item[(b)]
	If $n$ is even, then the restriction of $\chi$ to $\{[X]\,|\,0\neq X\in\mcc^0\}$ is positive definite.
	\end{itemize}
\end{corollary}
Motivated by the above results, we pose the following questions.
\begin{question}
	Let $\Lambda$ be an $n$-hereditary algebra.
	\begin{enumerate}
		\item 
		Assume that the restriction of $\chi$ to $\{[X]\,|\, X\in\ind\mcc^0\}$ is positive definite. Can we conclude that $\Lambda$ is $n$-representation finite?
		\item
		Assume that the restriction of $\chi$ to $\{[X]\,|\,0\neq X\in\mcc^0\}$ is positive definite, for even $n$. Can we conclude that $\Lambda$ is $n$-representation finite?
	\end{enumerate}	
\end{question}
In Theorem \ref{form}, we showed that if $\Lambda$ is $n$-representation infinite with odd $n$, then the restriction of the quadratic form $\chi$ to $\mathrm{K}_0(\mcc^0)$ is not positive definite. As a final result of this section, we show that for $n$-representation tame case, the restriction of $\chi$ to $\{[X]\,|\,0\neq X\in\mathscr{R}\}$ is not positive definite.
\begin{proposition}
If $\Lambda$ is an $n$-representation tame $K$-algebra with odd $n$ and $\rm{char}(K)\neq 2$, then the restriction of $\chi$ to $\{[X]\,|\,0\neq X\in\mathscr{R}\}$ is not positive definite.
\end{proposition}		 
\begin{proof}
 By Proposition \ref{tame}, $\mathscr{R}\neq 0$ and for any $X\in\mathscr{R}\,$ there exists a positive integer $l$ such that $\tau_n^l(X)\cong X$. Put $Y=X\oplus\tau_n(X)\oplus\tau_n^2(X)\oplus\cdots\oplus\tau_n^{l-1}(X)$. It is obvious that $\tau_n(Y)\cong Y$ and by Proposition \ref{tunu}, $\nu_n(Y)\cong Y$. We know that $\Phi$ gives the action of $\nu_n$ on the Grothendieck group ${\mathrm{K}}_0(\md_\Lambda)$. Therefore, $\underline{\dim}(\nu_nY) =\Phi(\underline{\dim} Y)$ and so $\Phi(\underline{\dim} Y)=\underline{\dim} Y$. By Lemma \ref{x-fi}$(\rm{a})$, we have
$$\langle \underline{\dim} Y,\underline{\dim} Y\rangle =- \langle  \underline{\dim}Y, \Phi(\underline{\dim} Y)\rangle =-\langle \underline{\dim} Y,\underline{\dim} Y\rangle.$$ Consequently, by Proposition \ref{th2}$(\rm{b})$, $\chi([Y])=q(\underline{\dim} Y)=0$.	
\end{proof}
\section*{Acknowledgments}
The research of the first author was in part supported by a grant from IPM (No.1403160036). The work of the second author is based upon research funded by Iran National Science Foundation (INSF) under project No. 4032107. Also, the research of the second author was in part supported by a grant from IPM (No. 1404160415).


\begin{thebibliography}{10}
\bibitem{ASS} \textsc{I. Assem, D. Simson and A. Skowro\'{n}ski}, \emph{Elements of the Representation Theory of Associative Algebras, Vol. 1: Techniques of Representation Theory}, London Math. Soc. Student Texts \textbf{65}, Cambridge University Press, Cambridge, 2006.
\bibitem{AP} \textsc{M. Auslander and M. I. Platzeck}, Representation theory of hereditary Artin algebras, in: \emph{Lecture Notes in Pure and Appl. Math., Vol. 37}, Dekker, New York, 1978, pp. 389--424.
\bibitem{AR1} \textsc{M. Auslander and I. Reiten}, Modules determined by their composition factor, \emph{Ill. J. Math.}, \textbf{29} (1985), 280--301.
\bibitem{ARS} \textsc{M. Auslander, I. Reiten and S. O. Smal{\o}}, \emph{Representation Theory of Artin Algebras}, Cambridge Studies in Advanced Mathematics \textbf{36}, Cambridge University Press, Cambridge, 1997.
\bibitem{AS1} \textsc{M. Auslander and S.O. Smal{\o}}, Preprojective modules over Artin algebras, \emph{J. Algebra}, \textbf{66}(1) (1980), 61--122.
\bibitem{Bas} \textsc{H. Bass}, \emph{Algebraic K-Theory}, W. A. Benjamin Inc., New York, 1968.
\bibitem{BGP} \textsc{I. N. Bernstein, I. M. Gelfand and V. A. Ponomarev}, Coxeter functors
and Gabriel’s theorem, \emph{Uspiehi Mat. Nauk}, \textbf{28} (1973), 19--33 (in
Russian), English translation in Russian Math. Surveys, \textbf{28} (1973),
17--32.
\bibitem{DIJ} \textsc{L. Demonet, O. Iyama and G. Jasso}, $\tau$-tilting finite algebras, bricks, and $g$-vectors, \emph{Int. Math. Res. Not. IMRN}, \textbf{3} (2019), 852--892.
\bibitem{DN1} \textsc{R. Diyanatnezhad and A. Nasr-Isfahani}, Relations for Grothendieck groups of n-cluster tilting subcategories, \emph{J. Algebra}, \textbf{594} (2022), 54--73.
\bibitem{DJY} \textsc{T. Dyckerhoff, G. Jasso and L. Yanki}, The symplectic geometry of higher Auslander
algebras: Symmetric products of disks, \emph{Forum of Mathematics, Sigma}, \textbf{9} (2021), e10.
\bibitem{G} \textsc{P. Gabriel}, Unzerlegbare Darstellungen I,  \emph{Manuscripta Math.}, \textbf{6} (1972), 71--103.
\bibitem{GLS} \textsc{C. Geiss, B. Leclerc and J. Schröer}, Rigid modules over preprojective algebras, \emph{Invent. Math.}, \textbf{165}(3) (2006), 589--632.
\bibitem{Ha} \textsc{D. Happel}, \emph{Triangulated Categories in the Representation Theory of Finite Dimensional Algebras}, London Math. Soc., \emph{Lecture Notes Series} \textbf{119}, Cambridge University Press, Cambridge, 1988.
\bibitem{HI1} \textsc{M. Herschend and O. Iyama}, $n$-representation-finite algebras and twisted fractionally Calabi-Yau algebras, \emph{Bull. Lond. Math. Soc.}, \textbf{43}(3) (2011), 449--466.
\bibitem{HI2} \textsc{M. Herschend and O. Iyama}, Selfinjective quivers with potential and $2$-representation-finite algebras, \emph{Compos. Math.}, \textbf{147}(6) (2011), 1885--1920.
\bibitem{HIMO} \textsc{M. Herschend, O. Iyama, H. Minamoto and S. Oppermann}, Representation theory
of Geigle-Lenzing complete intersections, \emph{Mem. Amer. Math. Soc.}, \textbf{285} (2023), Paper No. 1412.
\bibitem{HIO} \textsc{M. Herschend, O. Iyama and S. Oppermann}, $n$-representation infinite algebras, \emph{Adv. Math.}, \textbf{252}(3) (2014), 292--342.
\bibitem{HJ} \textsc{M. Herschend and P. Jørgensen}, Classification of higher wide subcategories for higher Auslander algebras
of type $A$, \emph{J. Pure Appl. Algebra}, \textbf{225}(5) (2021), 106583.
\bibitem{HZ1} \textsc{Z. Huang and X. Zhang}, Higher Auslander algebras admitting trivial maximal orthogonal subcategories, \emph{J. Algebra}, \textbf{330}(1) (2011), 375--387.
\bibitem{HZ3} \textsc{Z. Huang and X. Zhang}, The existence of maximal $n$-orthogonal subcategories, \emph{J. Algebra}, \textbf{321}(10) (2009), 2829--2842.
\bibitem{HZ2} \textsc{Z. Huang and X. Zhang}, Trivial maximal $1$-orthogonal subcategories for Auslander’s $1$-Gorenstein algebras, \emph{J. Aust. Math. Soc.}, \textbf{94}(1) (2013), 133--144.
\bibitem{I2} \textsc{O. Iyama}, Auslander correspondence, \emph{Adv. Math.}, \textbf{210}(1) (2007) 51--82.
\bibitem{I3} \textsc{O. Iyama}, Auslander-Reiten theory revisited, in: \emph{Trends in Representation Theory of Algebras and Related Topics}, 2008, pp. 349--398.
\bibitem{I4} \textsc{O. Iyama}, Cluster tilting for higher Auslander algebras, \emph{Adv. Math.}, \textbf{226}(1) (2011), 1--61.
\bibitem{I1} \textsc{O. Iyama}, Higher-dimensional Auslander--Reiten theory on maximal orthogonal subcategories, \emph{Adv. Math.}, \textbf{210}(1) (2007), 22--50.
\bibitem{IO} \textsc{O. Iyama and S. Oppermann}, $n$-representation-finite algebras and $n$-APR tilting, \emph{Trans. Amer. Math. Soc.}, \textbf{363}(12) (2011), 6575--6614.
\bibitem{IO1} \textsc{O. Iyama and S. Oppermann}, Stable categories of higher preprojective algebras, \emph{Adv. Math.}, \textbf{244}(10) (2013), 23--68.
\bibitem{IW1} \textsc{O. Iyama and M. Wemyss}, A new triangulated category for rational surface singularities, \emph{Ill. J. Math.}, \textbf{55}(1) (2011), 325--341.
\bibitem{IW2} \textsc{O. Iyama and M. Wemyss}, Maximal modifications and Auslander–Reiten duality for non-isolated singularities, \emph{Invent. Math.}, \textbf{197}(3) (2014), 521--586.
\bibitem{IW} \textsc{O. Iyama and M. Wemyss}, On the noncommutative Bondal–Orlov conjecture, \emph{Journal für die reine und angewandte Mathematik}, \textbf{683} (2013), 119--128.
\bibitem{J} \textsc{G. Jasso}, $n$-Abelian and $n$-exact categories, \emph{Math. Z.}, \textbf{283} (2016), 703--759.
\bibitem{JKM} \textsc{G. Jasso, B. Keller and F. Muro}, The Derived Auslander–Iyama correspondence, \href{https://arxiv.org/abs/2208.14413}{arXiv:2208.14413}, 2023.
\bibitem{Jo} \textsc{P. Jørgensen}, A torsion classes and t-structures in higher homological algebra, \emph{Int. Math. Res. Not.
IMRN}, \textbf{13} (2016), 3880--3905.
\bibitem{K} \textsc{B. Keller}, Deformed Calabi–Yau completions, \emph{J. Reine Angew. Math.}, \textbf{654} (2011), 125--180.
\bibitem{Ke} \textsc{O. Kerner}, Stable components of wild tilted algebras, \emph{J. Algebra}, \textbf{142} (1991), 37--57.
\bibitem{KP} \textsc{J. Kuzmanovich and A. Pavlichenkov}, Finite groups of matrices whose entries are integers, \emph{Amer. Math. Monthly}, \textbf{109}(2) (2002), 173--186.
\bibitem{Le} \textsc{Z. Leszczy\'{n}ski}, On the representation type of tensor product algebras, \emph{Fund. Math.}, \textbf{144}(2) (1994), 143--161.
\bibitem{M} \textsc{Y. Mizuno}, A Gabriel-type theorem for cluster tilting, \emph{Proc. London. Math. Soc.}, \textbf{108}(4) (2014), 836--868.
\bibitem{Mi} \textsc{Y. Mizuno}, Classifying $\tau$-tilting modules over preprojective algebras of Dynkin type, \emph{Math. Z.}, \textbf{277}(3) (2014), 665--690.
\bibitem{Ma} \textsc{K. Mousavand}, $\tau$-tilting finiteness of biserial algebras, \emph{Algebr. Represent. Theory}, \textbf{26}(6) (2023), 2485--2522.
\bibitem{OT} \textsc{S. Oppermann and H. Thomas}, Higher-dimensional cluster combinatorics and representation theory, \emph{J. Eur. Math. Soc.}, \textbf{14}(6) (2012), 1679--1737.
\bibitem{Re} \textsc{J. Reid}, Modules determined by their composition factors in higher homological algebra, \href{https://arxiv.org/abs/2007.06350}{arXiv:2007.06350}, 2020.
\bibitem{R} \textsc{C. M. Ringel}, Bricks in hereditary length categories, \emph{Resultate
der Mathematik}, \textbf{6}(1) (1983), 64--70.
\bibitem{Ri1} \textsc{C. M. Ringel}, Finite dimensional hereditary algebras of wild representation type, \emph{Math. Z}, \textbf{161} (1978), 235--255.
\bibitem{Ri} \textsc{C. M. Ringel}, Representations of $K$-species and bimodules, \emph{J. Algebra}, \textbf{41} (1976), 269--302.
\bibitem{Ri2} \textsc{C. M. Ringel}, The canonical algebras, \emph{in: Topics in algebra, Part 1 (Warsaw,
1988)}, 407--432, Banach Center Publications 26, Part 1, PWN, Warsaw, 1990. 
\bibitem{Ro} \textsc{J. J. Rotman}, \emph{An Introduction to Homological Algebra}, 2nd ed., Universitext, Springer, New York, 2009. 
\bibitem{SS} \textsc{D. Simson and A. Skowro\'{n}ski}, \emph{Elements of the Representation Theory of Associative Algebras, Vol. 2: Tubes and Concealed Algebras of
Euclidean type}, London Math. Soc. Student Texts \textbf{71}, Cambridge University Press, Cambridge, 2007.
\bibitem{SY} \textsc{A. Skowro\'nski and K. Yamagata}, \emph{Frobenius Algebras II, Tilted and Hochschild extension algebras}, EMS Textbooks Math., Eur. Math. Soc. Publ. House, Zürich, 2017.
\bibitem{W} \textsc{N. J. Williams}, New interpretations of the higher Stasheff-Tamari orders, \emph{Adv. Math.}, \textbf{407}(2022), Paper No. 108552.
\end{thebibliography}
\end{document}